\documentclass{amsart}

\usepackage{amsmath}
\usepackage{amssymb}
\usepackage{graphicx}
\usepackage[colorinlistoftodos]{todonotes}
\usepackage[colorlinks=true, allcolors=blue]{hyperref}
\usepackage{tikz}
\usepackage{endnotes}
\usepackage{mathtools}
\usepackage{multirow}
\usepackage{float}
\usepackage{lipsum}
\usepackage{esint}

\theoremstyle{definition}

 \newcommand{\comments}[1]{}

\newcommand{\Sp}{\mathbb S}
\newcommand{\R}{\mathbb{R}}
\newcommand{\C}{\mathbb{C}}

\newcommand{\Ord}{\text{Ord}}

\newcommand{\Wbar}{\overline{W}}
\newcommand{\mubar}{\overline{\mu}}
\newtheorem{theorem}{Theorem}[section]

\newtheorem{lemma}[theorem]{Lemma}
\newtheorem{remark}[theorem]{Remark}
\newtheorem{claim}[theorem]{Claim}
\newtheorem{proposition}[theorem]{Proposition}

\newtheorem{definition}[theorem]{Definition}
\newtheorem{example}{Example}[section]
\DeclarePairedDelimiter\ceil{\lceil}{\rceil}

\numberwithin{equation}{section}

\title[Rectifiable strata for harmonic maps to buildings]{Rectifiability of the singular strata for harmonic maps to Euclidean buildings}

\thanks{
CB supported in part by NSF DMS CAREER-1750254}

\author[Breiner]{Christine Breiner}
\address{Brown University\\
Department of Mathematics\\
Providence, RI}
\email{christine\underline{ }breiner@brown.edu}

\author[Dees]{Ben K. Dees}
\address{Brown University\\
Department of Mathematics\\
Providence, RI}
\email{benjamin\underline{ }dees@brown.edu}

\begin{document}
\maketitle
\begin{abstract}
    We define a natural notion of the singular strata for harmonic maps into $F$-connected complexes (which include locally finite Euclidean buildings), and prove the rectifiability of these strata.  We additionally establish bounds on the Minkowski content for certain quantitative strata, following the rectifiable Reifenberg program of \cite{nv17}.  This builds on a result of the second author \cite{dees}, which showed that the full singular set is $(n-2)$-rectifiable.
\end{abstract}

\section{Introduction}

The study of harmonic map to metric spaces with upper curvature bounds has been crucial in answering rigidity questions with application in algebraic geometry and geometric group theory. See for example \cite{gromov-schoen,kat97,katram98,linshaf,lasram96} for more historical results and applications and \cite{bdm,bddm,dm,dmw,bakker2024} for more recent work.  Gromov and Schoen \cite{gromov-schoen} initiated the study in order to prove  $p$-adic superrigidity for rank 1 symmetric spaces of non-compact type. Algebraic groups over $p$-adic fields act on {locally finite} Euclidean buildings by isometries; \cite{gromov-schoen} showed that equivariant harmonic maps to such targets possess enough regularity to invoke a Bochner formula, which proves such maps are constant.  Coupled with a result of Corlette \cite{cor}, this result extended the celebrated superrigidity (and arithmeticity) result of Margulis \cite{marg} to rank $1$ for both Archimedean and $p$-adic groups.  Recently, the authors, in collaboration with Mese \cite{bdm}, studied harmonic maps to non-locally finite Euclidean buildings. Algebraic groups over non-Archimedean fields with non-discrete valuation act on Euclidean buildings which are not locally finite. By proving equivariant harmonic maps to such targets are constant, \cite{bdm} extended the result of \cite{gromov-schoen} from $p$-adic superrigidity to non-Archimedean superrigidity for rank $\geq 1$. See also \cite{bf} for a dynamical approach for rank $\geq 2$.

In proving such rigidity statements via harmonic maps, it is often important to establish optimal regularity for the maps; in particular, the singular set of these maps must have Hausdorff codimension at least $2$. Recently, the second author improved this result for $F$-connected targets (which include locally finite Euclidean buildings) \cite{dees}. He showed that the singular set of a harmonic map to an $F$-connected complex is $(n-2)$-rectifiable. Here we further refine the regularity results and prove that in this setting the $k$-th stratum of the singular set is $k$-rectifiable.

\begin{theorem}\label{thm:k-rect}
For $X$ an $F$-connected complex, $\Omega\subset\R^n$ an open domain, and $u:\Omega\to X$ a harmonic map, the singular set $\mathcal{S}^k(u)$ is countably $k$-rectifiable.
\end{theorem}

\begin{remark}
    The authors expect that Theorem \ref{thm:k-rect} should also hold when $\Omega$ is a general Riemannian domain, with suitable technical alterations to the proof provided here.
\end{remark}

Classically, for a stationary harmonic map $u:M \to N$ between smooth Riemannian manifolds, where $N$ is compact with no boundary, the {regular set} $\mathcal R$ is defined by
\[
\mathcal R:= \{x \in M: \exists r>0 \text{ such that } u|B_r(x) \in C^0\}
\]and the {singular set} is given by $\mathcal S:= M \backslash \mathcal R$. (Note that $u$ continuous implies $u \in C^\infty$ in this case.)
One then defines the singular strata $\mathcal{S}^k$ by
\[
\mathcal{S}^k:=\{x\in\mathcal{S}:\text{no tangent map at }x\text{ is }(k+1)\text{-homogeneous}\}.
\]
In the current context, the regular set is defined differently; we want points in the regular set to be precisely those points about which the map does not ``see" the singular structure of the target. For harmonic $u:\Omega \to X$ where $X$ is an $F$-connected complex, we let the {\bf regular set} of $u$ be defined by 
\begin{align*}
\mathcal R(u):=\{x \in &\Omega: \exists r>0  \text{ such that }u(B_r(x)) \text{ maps into a totally geodesic }\\& \text{ subcomplex which is isometric to a subset of Euclidean space}\}.
\end{align*}Note that in this case, $u|{B_r(x)}$ can be regarded as an energy minimizer between Euclidean spaces and is therefore analytic on that ball. The {\bf singular set} of $u$ is again defined as the complement of $\mathcal R(u)$:
\[
\mathcal S(u):= \Omega \backslash \mathcal R(u).
\]
If one then attempts to define $\mathcal S^k(u)$ as in the classical setting, the examples in Section \ref{sec:sing strata examples} illustrate the difficulty. 
Indeed, a natural definition of the singular strata must take into account the way in which the image of a harmonic map admits a ``splitting," in the sense that we can ``factor off an $\R^j$," and the impact that will have on detecting homogeneity. 

Because the target is simplicial, away from the vertices there is a natural splitting of the complex itself. That is, we can always  split the image of $u$ at any point $x$ for which $u(x)$ is not a $0$-cell. If $u(x)$ lies in the interior of a $j$-cell, $j>0$, then there is a neighborhood $U$ of $u(x)$ which is isometric to a neighborhood in $\R^j\times Y^{N-j}$, where $Y$ is an $F$-connected complex of dimension $N-j$. Since $u$ is continuous by \cite[Theorem 2.4.6]{korevaar-schoen1}, there is a ball $B_\sigma(x)$ so that $u(B_\sigma(x))\subseteq U$ which means that we can write $u|B_\sigma(x)=(u_1,u_2):B_\sigma(x)\to\R^j\times Y^{N-j}$.

In addition to this natural splitting of the target, at domain points $x$ where $\Ord_u(x)=1$, \cite[Lemma 6.3.(ii)]{gromov-schoen} implies that the mapping splits in the following sense. If $u:\Omega \to X$ is harmonic and $\Ord_u(x)=1$ then there exists $\sigma>0$ such that
\[
u|B_\sigma(x)=(u_1,u_2):B_\sigma(x)\to \R^j\times Y\subseteq X
\]
where $u_1:B_\sigma(x)\to\R^j$ and $u_2:B_\sigma(x)\to Y$ are both harmonic maps.  The map $u_1$ has rank $j$ at each point of $B_\sigma(x)$, and $\Ord_{u_2}(x)>1$.  This splitting is uniquely determined by $u$ and $x$.

The trouble with attempting to use the classical definition is as follows. If, around a singular point, the harmonic map ``splits" in the sense described above, then a tangent map might only detect the {\it regular} part of the mapping. Thus, one cannot expect to determine the right stratum for the singular point from the homogeneity of its tangent map. Our definition for $\mathcal S^k(u)$ (cf. Definition \ref{def:strata}) provides the most natural way to appropriately detect the correct homogeneity. We initially consider only those singular points around which there is no splitting. For singular points that admit a splitting, we consider the stratum for only the map $u_2$ and ignore the regular part $u_1$. 

 Theorem \ref{thm:k-rect} follows from proving $k$-rectifiability for the set of singular points with no splitting (cf. Theorem \ref{thm:k-rect-quant}) and using this reduction for other singular points. 
The rectifiability proof follows from the covering arguments first developed by Naber and Valtorta \cite{nv17,nv18} for harmonic and approximate harmonic maps, also using the form refined by De Lellis~ et.~al.~\cite{dmsv} for $Q$-valued functions and some techniques therein. 
One notable technical change is in defining the quantitave strata (cf. Definition \ref{def:quant-hom})
Since our harmonic maps possess uniform Lipschitz bounds, we use $C^0$ distance rather than $L^2$ distance to compare our maps to $k$-homogeneous ones. As a result of this new definition, for certain cone-splitting arguments (cf. Section \ref{sec:conesplittingtranslating}), we will need a sequence of homogeneous maps to converge uniformly. This presents an additional minor technical challenge which is overcome using homogeneity and properties of the order function for harmonic maps (cf. Lemmas \ref{lem:hom-grad-bound}, \ref{lem:hom-ord-bound}).

The stratification used in this paper broadly resembles the ``factorization" used to define strata for $Q$-valued maps given in \cite[Section 5.3]{fms}.  We expect, using the techniques outlined here, that the analogue of Theorem \ref{thm:k-rect} will hold for those strata as well.

\subsection{Outline of the paper}

The body of this work is structured as follows:
\begin{itemize}
    \item In Section \ref{sec:prelim}, we first collect the definitions and previously-known results which are relevant to our problem.  This includes the definitions of $F$-connected complexes and harmonic maps, and the appropriate notion of a tangent map.  We also define the order function and the smoothed order function, homogeneity, and $k$-homogeneity. 
    \item In Section \ref{sec:sing strata}, we define the singular strata in our setting. We begin by defining what it means for a map to ``split" at a point $x$. Section \ref{sec:sing strata examples} provides a few illustrative examples to motivate our definition.  In the final two subsections we define the strata and the quantitative strata.
    \item In Section \ref{sec:results} we state the main results of the paper, first local (conical) versions and then global versions. All of these results are stated for singular points with no splitting. The proof of Theorem \ref{thm:k-rect} is contained in Section \ref{sec:proof of main}.
    \item The results in Section \ref{sec:results} are proved using the framework of \cite{nv17}; to apply this framework we need several ``cone splitting" results, which are in Section \ref{sec:conesplitting}, and we need a result linking our monotone quantity to the mean flatness, which is supplied in Section \ref{sec:meanflat}.
    \item In Section \ref{sec:sketchy-proof}, we prove the main (local) theorems of Section \ref{sec:results}.
    \item In Appendix \ref{sec:arxivproof}, we provide a proof of a technical ``intermediate covering", Proposition \ref{prop:initial-cover} and a proof of the main covering theorem, Theorem \ref{thm:k-2nd-cover}. 
\end{itemize}

\section{Preliminaries}\label{sec:prelim}

\subsection{CAT(0) spaces}\label{CAT0sec}
A CAT(0) space $(X,d)$ is a geodesic space of non-positive curvature, where curvature is defined by triangle comparison. Particular examples of CAT(0) spaces include $F$-connected complexes (the focus of this paper) as well as Hadamard manifolds.  We refer to \cite{bridson-haefliger} for a complete introduction to these spaces.  As we work only with simplicial complexes, a complete exposition is not necessary for this paper.

\subsection{$F$-connected complexes}
\begin{definition}\label{def:fcon}
A CAT(0) simplicial complex $(X,d)$ of dimension $N$ is said to be {\bf $F$-connected} if:
\begin{enumerate}
\item Each simplex of $X$ is a Euclidean simplex; i.e. is isometric to the image of a standard Euclidean simplex
\[
\Delta^\ell:=\{(x^0,\dots,x^\ell):x^0+\dots+x^\ell=1\}
\]
(where $0\leq \ell\leq N$) under an invertible affine linear transformation.
\item $X$ is a ``homogeneous" complex in the sense that every $\ell$-simplex, $\ell\in \{0, \dots, N-1\}$, is the face of a simplex of dimension $\ell+1$.
\item Any two adjacent simplices $S,S'$ of $X$ are contained in a totally geodesic subcomplex $X_{S,S'}$ which is isometric to a subset of the Euclidean space $\R^N$.
\end{enumerate}

(We say that two simplices of $X$ are adjacent if their intersection is nonempty.  A subcomplex of $X$ is the union of some of the simplices of $X$.  A subcomplex $Y$ of $X$ is said to be totally geodesic if, for any two points $x,y\in Y$, the geodesic between $x$ and $y$ is contained in $Y$ as well.)

An {\bf $m$-flat} of $X$ is a totally geodesic subset of $X$ isometric to $\R^m$.

For a point $P\in X$, where $X$ is $F$-connected, we denote the {\bf tangent cone at $P$} by $X_P$.  Observe that a neighborhood of $P$ in $X$ is isometric to a neighborhood of the origin in $X_P$.  For this reason, we will sometimes blur the distinction between flats of $X_P$ and flats of $X$ which contain $P$.

\end{definition}

A significant class of $F$-connected complexes are the locally finite Euclidean buildings.  To realize a Euclidean building as an $F$-connected complex, one subdivides the chambers of the building into simplices as necessary.  For a metric geometer's introduction to Euclidean buildings, we refer the reader to \cite{kleiner-leeb}.

Since $F$-connected complexes are locally conical and our analysis is alway local, we will restrict our attention to conical complexes.
\begin{definition}
    A CAT(0) simplicial complex $(X_C,d)$ is a {\bf conical $F$-connected complex} of dimensions $N$ with cone point $0_X$ if:
    \begin{enumerate}
        \item Each $\ell$-simplex of $X_C$, except the vertex $0_X$, is a Euclidean simplicial cone; i.e. the image of $\{(x^1,\dots,x^\ell): x^i\geq0\}$ under an invertible affine linear transformation.
        \item $X_C$ is homogeneous in the sense that every $\ell$-simplex, $\ell\in\{0, \dots, N-1\}$, is the face of a simplex of dimension $\ell+1$.
        \item Any two simplices $S,S'$ of $X_C$ are contained in a totally geodesic subcomplex $X_{S,S'}$ which is isometric to a subset of $\R^N$.
    \end{enumerate}
\end{definition}
\subsection{Harmonic maps}
Since $F$-connected complexes are CAT(0) spaces, harmonic (or ``energy minimizing") maps into these exist and are generally well-behaved. The theory of harmonic maps into CAT(0) metric spaces which we will follow in this paper was developed in \cite{gromov-schoen}, \cite{korevaar-schoen1}, and \cite{korevaar-schoen2}; we refer the interested reader to these for full details.  These papers introduce the Sobolev spaces $W^{1,2}(\Omega,X)$, where $\Omega$ is a Euclidean domain and $X$ a CAT(0) metric space, and generalize the energy density function $|\nabla u|^2$ to these spaces as well. (See also \cite{jost} for a construction of equivariant harmonic maps between symmetric spaces of non-compact type and \cite{heinonen} and the references therein for alternate definitions of Sobolev spaces.) They also generalize the trace to this setting, allowing us to speak meaningfully of the boundary values of $u\in W^{1,2}(\Omega,X)$.  If two functions $u,v\in W^{1,2}(\Omega, X)$ have the same trace, we write $u=v$ on $\partial\Omega$, as usual.

\begin{definition}
    For a map $u\in W^{1,2}(B_r(x),X)$ the {\bf energy of $u$ on $B_r(x)$} is
    \[
   E_u(x,r):=\int_{B_r(x)}|\nabla u(y)|^2dy.
    \]

    We say that a map $u\in W^{1,2}(B_r(x),X)$ is {\bf harmonic} (or energy-minimizing) if for every $v$ so that $u=v$ on $\partial B_r(x)$,
    \[
    E_u(x,r)\leq E_v(x,r).
    \]
    A map $u\in W^{1,2}(\Omega,X)$ is {\bf harmonic} if for every $x\in \Omega$ there exists an $r>0$ such that $u|B_r(x)$ is harmonic by the previous definition.
\end{definition}

By \cite[Theorem 2.4.6]{korevaar-schoen1}, a harmonic map $u:\Omega\to X$ into a CAT(0) space, and in particular into an $F$-connected complex $X$, is continuous and moreover locally Lipschitz continuous, with Lipschitz constant at $p\in\Omega$ bounded in terms of the dimension $n$ of the domain, the total energy of $u$, and the distance from $p$ to $\partial\Omega$.  

Because $F$-connected complexes are locally conical and harmonic maps into them are continuous, if $u:\Omega\to X$ is such a map into an $F$-connected complex, we can analyze it {\em locally} as a map into a conical complex.  Precisely, for each $p\in\Omega$, there is a neighborhood $V\subseteq X$ of $u(p)$ isometric to a neighborhood $U\subseteq X_{u(p)}$ of the tangent cone at $u(p)$, and a ball $B_r(p)$ so that $u(B_r(p))\subseteq V$.
 
\subsection{Order and smoothed order}\label{subsec:tangent-maps}

To define the lower strata appropriately, indeed to understand the subtleties involved, one first needs to understand the monotone quantity and properties of tangent maps in this setting.

\begin{definition}\label{def:terminology}For a map $u:\Omega\subset\R^n\to X$, where $X$ is an $F$-connected complex, let
\begin{align}
\nonumber I(x,r)&=\int_{\partial B_r(x)}d^2(u(y),u(x))dy\\
\nonumber\Ord(x,r)&=\frac{rE(x,r)}{I(x,r)}.
\end{align}
If there are multiple maps $u_i$ under consideration, we use a subscript, i.e. $I_{u_i}$, to disambiguate.
\end{definition}
    In \cite{gromov-schoen} it is shown that for each $x$, $\Ord(x,r)$ is a monotonically increasing function of $r$. Therefore, the following definition makes sense. \begin{definition}The {\bf order function} $\Ord:\Omega \to [1,\infty)$ is given by
    \[
    \Ord(x):=\lim_{r\to0}\Ord(x,r).
    \]
    
\end{definition}

For technical reasons, in our analysis of maps into conical $F$-connected targets, it is convenient to consider the {\bf smoothed order}. Fix once and for all $\phi:\R_{\geq0} \to \mathbb R$ such that $\phi\equiv 1$ on $[0,0.5]$, $\phi(x)=2-2x$ on $[0.5,1]$, and $\phi = 0$ otherwise. Then for a harmonic $u:\Omega \subset \mathbb R^n \to X_C$, define
\begin{align}
 \nonumber E_\phi(x,r)&=\int_{\R^n}|\nabla u(y)|^2\phi\left(\frac{|y-x|}{r}\right)dy\\
\nonumber I_\phi(x,r)&=-\int_{\R^n} d^2(u(y),0_X)|y-x|^{-1}\phi'\left(\frac{|y-x|}{r}\right)dy\\
\nonumber\Ord_\phi(x,r)&=\frac{rE_\phi(x,r)}{I_\phi(x,r)}
\\ 
\label{def:smooth-ord}
\Ord_\phi(x,0)&:= \lim_{r\to 0} \Ord_\phi(x,r).\end{align}When we need to distinguish between various maps, we add the map to the subscript, e.g. $E_{\phi,u}$. Like the classical order function, the smoothed order $\Ord_\phi(x,r)$ is monotone nondecreasing in $r$ (cf. \cite[Proposition 4.1]{dees} for a proof).

Additionally, we have the following technical identities, also proved in \cite[Proposition 4.1]{dees}:
\begin{align}
    \label{eq:energy-deriv}\partial_rE_\phi(x,r)&=\frac{n-2}{r}E_\phi(x,r)-\frac{2}{r^2}\int_{\R^m}|\partial_{\nu_x}u(y)|^2|y-x|\phi'\left(\frac{|y-x|}{r}\right)dy\\
    \label{eq:height-technical}s^{1-n}I_\phi(x,s)&=r^{1-n}I_\phi(x,r)\exp\left(-2\int_s^r\Ord_\phi(x,t)\frac{dt}{t}\right).
\end{align}
In application, these identities are useful for establishing greater control on the energy and height of $u$ at $x$.  In the presence of Lipschitz bounds on $u$, the former identity gives good control on the radial derivative of $E_\phi$.  In the presence of bounds on $\Ord_\phi$, the latter identity lets us bound (say) $I_\phi(x,4)$ in terms of $I_\phi(x,1)$.  It also establishes the monotonicity of the scale-invariant height $r^{1-n}I_\phi(x,r)$.  In fact, more can be said---the unsmoothed height $I(0,r)$ has the same scale-invariant monotonicity.  Precisely, if $0<s<r$,
\[
s^{1-m}I(0,s)\leq r^{1-m}I(0,r)
\]
(cf. \cite[Equation (4.11)]{dees}).

We will often be interested in sequences of harmonic maps. To guarantee that a subsequence converges, the maps of interest $u:B_s(0) \to X_C$ will satisfy the following condition:
\begin{equation}\tag{$\star$}\label{eq:Icond}
\text{Ord}_\phi(0,s) \leq \Lambda. 
\end{equation}Note that the $s>0$ will be fixed in each context but may vary from one application to the next.

\subsection{Tangent maps and homogeneity}
Gromov and Schoen, in \cite{gromov-schoen}, develop a notion of tangent maps for harmonic maps into $F$-connected complexes. (The situation for more general CAT(0) space targets is developed in \cite{korevaar-schoen1,korevaar-schoen2} and extended to CAT(1) targets in \cite{BFHMSZ}.) Tangent maps are homogeneous but, in contrast to the smooth setting, are non-constant even at smooth points. In particular, at points of order 1, the tangent map is a non-constant affine map. 

Following \cite{gromov-schoen}, let $u:B_\sigma(x) \subset \mathbb R^n \to (X,d)$ be a harmonic map into an $F$-connected complex $X$ and change coordinates in the domain so that $x=0$. Then for each $\lambda>0$ define the map 
\begin{equation}\label{eq:rescalings}
u_{\lambda}:B_{\sigma/\lambda}(0) \to (X,d_\lambda)
\end{equation}
by $u_\lambda(y) := u(\lambda y)$, where $d_\lambda =(\lambda^{1-n}I(0,\lambda))^{-1/2}d$. 

The scalings on domain and target are set up so that for sufficiently small $\lambda>0$ the functions $u_\lambda$ have uniform energy bounds on $B_1$ and by the regularity theory of \cite{gromov-schoen}, a subsequence converges uniformly to a non-constant, harmonic map $u_{*}:B_1 \to X_{u(x)}$, where $X_{u(x)}$ denotes the tangent cone of $X$ over $u(x)$. 
\begin{definition}\label{def:tangentmap}
    Any subsequential limit $u_{*}:B_1 \to X_{u(x)}$ realized by the rescaling process above is called a {\bf tangent map to $u$ at $x$}. 
\end{definition}

By construction, any tangent map $u_{*}$ is homogeneous of degree $\alpha=\text{Ord}_u(x).$ Because the tangent maps of a harmonic map are maps into $X_{u(x)}$, to define homogeneity precisely we need only consider maps into conical complexes.  By scaling with respect to the cone point, in conical $F$-connected complexes, we can define homogeneity extrinsically, rather than intrinsically as in \cite{gromov-schoen}.
\begin{definition}Let $X_C$ be a conical $F$-connected complex with cone point $0_X$.
    We say that a map $h:\mathbb R^n\to X_C$ is {\bf homogeneous of degree $\alpha$ about $x_0$} if $h(x_0)=0_X$ and for all $z\in\mathbb R^n$ and $\lambda>0$
\begin{equation}\label{eq:hom-def}
h(x_0+\lambda z)= \lambda^\alpha h(x_0+z).
\end{equation}
When $x_0=0$ we simply say $h$ is homogeneous of degree $\alpha$.

\end{definition}

In \eqref{eq:hom-def}, we interpret the scalar multiplication on the right-hand side to mean scaling the point $h(x_0+z)$ by a factor of $\lambda^\alpha$ about the cone point $0_X$.  (Formally, it is the unique point at distance $\lambda^\alpha d(0_X,h(x_0+z))$ along the geodesic ray from $0_X$ to $h(x_0+z)$, where the uniqueness of this geodesic ray is a consequence of conicality.)

When we want to restrict our domain, i.e. $h:B_r(x_0) \to X_C$, we say $h$ is homogeneous of degree $\alpha$ on $B_r(x_0)$ if \eqref{eq:hom-def} holds for every $z,\lambda$ where expressions on both the left and right of \eqref{eq:hom-def} make sense.

\begin{definition}
    A homogeneous degree $\alpha$ map $h:\R^n\to X_C$ is  {\bf $k$-homogeneous about the point $x_0$} if it is homogeneous about the point $x_0$ and there is a $k$-dimensional subspace $V\subseteq\R^n$ so that
    \[
    h(x+v)=h(x)
    \]
    for all $x\in\R^n$ and all $v\in V$.

\end{definition}
Note that for all $v \in V$, one necessarily has that $h(x_0+v) = 0_X$. Moreover, by repeated applications of \cite[Lemma 7.4]{dees}, if $h$ is homogeneous about the points $x_0,x_1,\dots,x_k$, where $\{x_i-x_0\}_{i=1}^k$ are linearly independent, then $h$ is $k$-homogeneous, and a subspace witnessing $k$-homogeneity is $V=\text{span}\{x_i-x_0\}_{i=1}^k$.  

\section{The Singular Strata}\label{sec:sing strata}

To characterize the splittings systematically we introduce the following definitions.

\begin{definition}
    For a map $u:\Omega\to X$, we say that $u$ {\bf splits on $B_\sigma(x)$ with $\R^j$ factor} if there is an $F$-connected complex $Y$ so that
    \[
    u|B_\sigma(x)=(u_1,u_2):B_\sigma(x)\to\R^j\times Y\subseteq X
    \]
    where $\R^j\times Y$ is an isometric totally geodesic subcomplex of $X$.  Moreover, we call $(\sigma,u_1, u_2, j, Y)$ {\bf splitting data for $u$ at $x$}.
\end{definition}
 \noindent   Observe that every map (trivially) splits with $\R^0$ factor (with $Y=X$).

    For $x\in\Omega$, and each $j\in\{0,1,\dots,N\}$, consider
    \begin{equation}\label{eq:sigmaj}
    \sigma_j(x):=\sup\{\sigma\geq0:u\text{ splits on }B_\sigma(x)\text{ with }\R^j\text{ factor}\}.
    \end{equation}
    If $\sigma_j=0$, there is no nontrivial ball $B=B_\sigma(x)$ on which $u|B$ splits with $\R^j$ factor. Let
    \begin{equation*}
   J(x):= \max\{j \in \{0, \dots, N\}: \sigma_j(x) >0\}.
   \end{equation*}
  Then, there exist $(u_1, u_2):B_{\sigma_J(x)}(x) \to \R^J\times Y$ such that $u|B_{\sigma_J(x)}(x)=(u_1, u_2)$.
\begin{definition}\label{def:splits}
    
    We call $(\sigma_{J(x)},u_1,u_2,J(x),Y)$ {\bf optimal splitting data for $u$ at $x$}, and we say that {\bf $u$ splits at $x$ with optimal $\R^{J(x)}$ factor}.
\end{definition}

Observe that $x$ is a regular point for $u$ if and only if $J(x)=N$. Alternately, if $J(x)=0$ then $\Ord_u(x)>1$ and $u(x)$ is a $0$-cell of $X$.

\subsection{Motivating Examples}
\label{sec:sing strata examples}
  We shall see that Euclidean factors in the target space should be disregarded when studying the singular set, which motivates our definition of the singular strata (cf. Definition \ref{def:strata}).  

We begin with some informative examples.  While all of these examples have $F$-connected targets, we do not explicitly give the simplicial structures witnessing this fact, as doing so would obscure the behavior we wish to highlight.

\begin{example}
    
The standard singularity model for harmonic maps into buildings is a map from the complex disc to a tripod, $f_{Y}:B_1(0)\to Y$, outlined in e.g. the introductions of \cite{gromov-schoen,sun}.  The tripod $Y$ consists of three rays meeting at a common point $0_Y$. Divide the disc into three sectors, bounded by rays from $0$ to $e^{i\pi/3}$, $e^{-i\pi/3}$, and $-1$.  The map takes each of these sectors to one of the rays of $Y$, with $f_Y(0)=0_Y$.  It satisfies
\[
d(f_Y(re^{i\theta},0_Y))=r^{3/2}\cos(3\theta/2)
\]
and hence we often think of this map as ``$\text{Re}(z^{3/2})$" as a function from $B_1(0)\subset\C$ to $Y$. This map is regular everywhere except $0$, because this is the only point at the boundary of all three sectors.  At $0$, the order of this map is $\frac32$, as one might expect. See Figure \ref{fig:tripod} for a rough picture with a few representative level sets of $f_Y$ marked out. Note that in this example, $J(0)=0$ as neither the map nor the target split (or in our parlance, the map splits with optimal $\R^0$ factor).

\begin{figure}
\begin{center}
\begin{tikzpicture}
\draw[black](0,0) circle (2);
\draw[black,very thick](1,2*0.866)--(0,0);
\draw[black,very thick](1,-2*0.866)--(0,0);
\draw[black,very thick](-2,0) node [anchor=east] {$f_Y^{-1}(0_Y)$} --(0,0);
\draw[cyan,very thick] (1.5,1.323) node [anchor=west] {$f_Y^{-1}(P_1)$} .. controls (0.75,0) .. (1.5,-1.323) ;
\draw[magenta,very thick] (-1.5*0.5-1.323*0.866,-0.5*1.323+1.5*0.866)  .. controls (-0.5*0.75,0.75*0.866) .. (-1.5*0.5+1.323*0.866,0.5*1.323+1.5*0.866) node [anchor=south] {$f_Y^{-1}(P_2)$};
\draw[olive,very thick] (-1.5*0.5-1.323*0.866,0.5*1.323-1.5*0.866) .. controls (-0.5*0.75,-0.75*0.866) .. (-1.5*0.5+1.323*0.866,-0.5*1.323-1.5*0.866) node [anchor=north] {$f_Y^{-1}(P_3)$};

\draw[->] [black] (3-0.5,0) .. controls (4-0.75,0.5) .. (5-1,0);
\draw[black] (4-0.75,-0.25) circle (0pt) node [anchor=south]{$f_Y$};

\draw[black](6-1,0)--(8-1,0);
\draw[black](6-1,0)--(6-1-1,2*0.866);
\draw[black](6-1,0)--(6-1-1,-2*0.866);
\filldraw[cyan] (7.5-1,0) circle (2pt) node [anchor=north] {$P_1$};
\filldraw[magenta] (6-1-0.5*1.5,1.5*0.866) circle (2pt) node [anchor=west] {$P_2$};
\filldraw[olive] (6-1-0.5*1.5,-1.5*0.866) circle (2pt) node [anchor=west] {$P_3$};
\filldraw[black] (6-1,0) circle (2pt) node [anchor=east] {$0_Y$};

\end{tikzpicture}
\caption{The domain of $f_Y$ is shown on the left as a disc; the target tripod on the right. Four points on the tripod are marked---one on each leg, and the common point $0_Y$. Their respective preimages are labeled in the domain.}
\label{fig:tripod}
\end{center}

\end{figure}
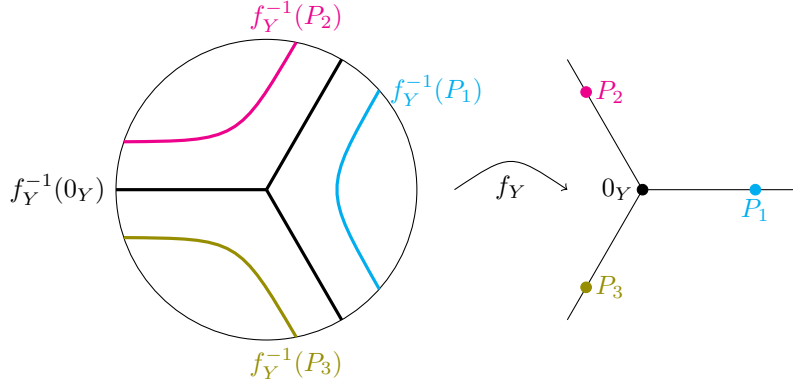
\end{example}

The next two examples demonstrate how target splittings allow the regular part of a harmonic map to overwhelm the singular part when limiting to the tangent map. We first consider the target splitting at an order one point, where we know splitting is guaranteed by \cite[Theorem 6.3]{gromov-schoen}.
\begin{example}
    Consider the map $u\colon (0,1)\times B_1(0)\to \R\times Y$ defined by $u(t,z)=(t,f_Y(z))$ where $f_Y$ is as in the previous example. Figure \ref{fig:tricyl} sketches the behavior of this map.  Each horizontal ``slice" of the domain is mapped to the corresponding horizontal ``slice" of the target complex by the map $f_Y$. The map $u$ splits at $0$ with optimal $\R^1$-factor where $u_1=(t,0)$ and $u_2 = (0, f_Y)$ and the singular set of $u$ is precisely $(0,1)\times\{0\}$. Our intuition tells us that this set should be the $1$-stratum $\mathcal{S}^1(u)$. 

However, the tangent map to $u$ at any point in $(0,1)\times\{0\}$ is $u_*(t,z)=\left(\sqrt{\frac3{4\pi}}t,0\right)$, a linear map of rank $1$ from $\R^3$ to $\R\times Y$.  Since $u_*$ is independent of $z$, by definition it is $2$-homogeneous and thus the classical definition would imply that any point in $(0,1)\times\{0\}$ should lie in the $2$-stratum.  This does not agree with our intuition that the singular points of $u$ should lie in the $1$-stratum.

The reason for this is that while the {\em singular behavior} of $u$ results from the $u_2$ factor, the {\em tangent map at singular points} is determined by the regular factor $u_1$.  The solution is to ignore the regular factor, and consider the tangent map of only the singular factor.  When we do so, we observe that the map $(t,z)\mapsto f_Y(z)$ has one direction of constancy (the $t$-axis of the domain), and conclude that the singular points are in the $1$-stratum, as expected.

\begin{figure}
    \begin{center}
\begin{tikzpicture}

\draw[very thick] (0,0) ellipse (2 and 1);
\draw[thick] (-2,0)--(-2,5);
\draw (2,0)--(2,5);
\draw (0,0)--(0,5);
\draw[thick] (-2,0)--(0,0);
\draw[thick] (-2,5)--(0,5);
\draw (0,0)--(0.9,0.916)--(0.9,5.916)--(0,5);
\draw[very thick] (0,0)--(1.1,-0.816)--(1.1,5-0.816)--(0,5);
\draw[dashed] (0,2.8) ellipse (2 and 1);
\draw[dashed] (0,2.8)--(-2,2.8);
\draw[dashed] (0,2.8)--(.9,2.8+0.816);
\draw[dashed] (0,2.8)--(1,2.8-0.816);

\draw[very thick] (0,5) ellipse (2 and 1);
\draw[dashed] (0,0) ellipse (2 and 1);

\draw(7,0)--(7,5);
\draw(9,0)--(9,5);
\draw(9,0)--(7,0);
\draw(9,5)--(7,5);
\draw(7,5)--(6.1,5.916);
\draw (7,5)--(5.9,5-0.816)--(5.9,-0.816)--(7,0);
\draw (7,0)--(6.1,0.916);
\draw (6.1,5.916)--(6.1,0.916);
\draw[dashed] (7,2.8)--(9,2.8);
\draw[dashed] (7,2.8)--(6.1,2.8+0.916);
\draw[dashed] (7,2.8)--(5.9,2.8-0.816);

\draw[->] [black] (3,2.5) .. controls (4,3) .. (5,2.5);
\draw[black] (4,2.25) circle (0pt) node [anchor=south]{$u$};

\end{tikzpicture}
\end{center}
\caption{The domain of $u$ is sketched on the left, and the target complex on the right.  The domain is separated into three ``sectors," corresponding to the preimages of the three half-planes which comprise $Y\times I$.}
\label{fig:tricyl}
\end{figure}
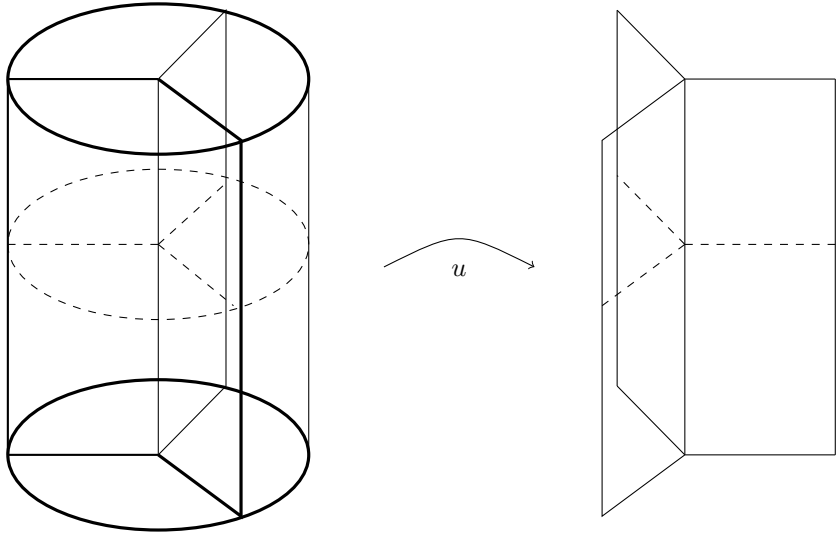

\end{example}
The final example shows how target splitting, even at a higher order point, can still make a classical definition of the singular stratum unnatural.

\begin{example}\label{ex:example3}
  Let $g:\R^4\to\R$ such that $g(r,\Theta) = r^2 \phi(\Theta)$ and $\Delta_{\mathbb S^3} \phi(\Theta) = 8 \phi(\Theta)$. Then, $g$ is a homogeneous harmonic polynomial of degree $2$, with a unique point of order $2$ at $0$. Define $f_K:B_1(0)\to K$ to be a map from the disk into a $5$-pod, analogous to the definition of $f_Y$. Now consider the map $h(r,\Theta,z)=(g(r,\Theta),f_K(z))$. Observe that $h$ splits at $0$ with optimal $\R^4$-factor, and $u_1=(g,0)$ while $u_2=(0,f_K)$. The singular set of $h$ is $\R^4\times\{0\}$, which suggests that singular points should be in the $4$-stratum.

   However, a tedious calculation demonstrates that at the origin, the tangent map $h_*$ is a non-zero scalar multiple of $u_1=(g,0)$. Thus it is homogeneous degree $2$ and comes from the regular factor.  Moreover, this tangent map is $2$-homogeneous but not $3$-homogeneous, suggesting that (classically) the origin should lie in the $2$-stratum. But we know it should lie in the $4$-stratum.
\end{example}

\subsection{Defining the Strata}
With these examples in mind, we first characterize singular points by their maximal target $\R^j$ factor. 
\begin{definition} Let $u:\Omega \to X$ and for $j \in \{0, \dots, N-1\}$, define
\[
F_j(u):=\{ x \in \mathcal S(u) : u \text{ splits at }x \text{ with optimal } \R^j\text{ factor}\}
\]in the sense of Definition \ref{def:splits}.
    
\end{definition}
Observe that a point $x \in F_j(u)$ has an optimal target $\R^j$ factor but no target $\R^{j+1}$ factor. In particular, 
for $x\in F_{j}(u)$, there exists splitting data $(\sigma, u_1, u_2, j, Y)$ so that $u|B_\sigma(x)=(u_1,u_2):B_\sigma(x)\to\R^j\times Y$ and $x\in F_0(u_2)$.

\begin{definition}\label{def:strata}
    Let $u:\Omega\to X$, where $X$ is an $F$-connected complex, $\Omega\subseteq\R^n$, and $u$ is harmonic. 
For each $k \in \{0, \dots, n-2\}$, let \[
\mathcal{S}^k_0(u):=\{x\in F_0(u):\text{no tangent map at }x\text{ is }(k+1)\text{-homogeneous}\}.
\]
For $j \in \{1, \dots, N-1\}$, recalling \eqref{eq:sigmaj}, let
\begin{align*}
\mathcal{S}^k_j(u):=\{x\in F_j(u):&\text{ for all splittings }(\sigma, u_1^\sigma, u_2^\sigma, j, Y_\sigma ), \text{ with }\sigma \in (0, \sigma_j(x)],\,x\in\mathcal{S}^k_0(u_2^\sigma)\}.
\end{align*}

Finally, define 
\[\mathcal{S}^k(u):=\bigcup_{j=0}^{N-1}\mathcal{S}^k_j(u).
\]
\end{definition}
 
We shall prove our theorems for ${\mathcal{S}}_0^k(u)$ and recover conclusions for the full singular set simply from the definitions and splitting phenomena.  The discussion at the beginning of this section implies that if $x\in\mathcal{S}_0^k(u)$, then $\Ord_u(x)>1$ and $u(x)$ is a $0$-cell of $X$.  In particular, when we restrict attention to a conical complex, because conical complexes have only one $0$-cell, the cone point $X_C$, we shall have that $\mathcal{S}^k_0(u)\subseteq u^{-1}(0_X)$.

\subsection{Almost Homogeneity and Quantitative Strata}

Since we are interested in quantitative results, we need a notion of almost $k$-homogeneity.
\begin{definition}\label{def:quant-hom}
We say that a map $u$ is {\bf $(\eta,r,k)$-homogeneous at $x$} if there exists a $k$-homogeneous map $h$ 
so that
\begin{equation}\label{eq:khom}
\sup_{B_r(x)} d(u,h^x)\leq \eta\left(\frac{I_\phi(x,r)}{r^{n-1}}\right)^{\frac{1}{2}}
\end{equation}where $h^x(y):=h(y-x)$ is defined to center the homogeneity of $h$ at $x$.

\end{definition}
The function of $r$ on the right hand side is present so that the inequality is preserved under {\em both} domain and target rescalings. (Note that the canonical homogeneous competitor is a tangent map of $u$, which is realized as a limit under domain and target rescalings.)

\begin{definition}\label{def:local-quant-strat}
For $\eta>0$, $r \in (0,1)$, let $${\mathcal S}^k_{0,\eta,r}(u):=\{x \in \mathcal{S}_0^k(u): u \text{ is not } (\eta, s, k+1)\text{-homogeneous at }x\text{ for any }s \in [r,1]\}.$$
For $\eta>0$, let $$\mathcal{S}^k_{0,\eta}(u):= \{x \in \mathcal{S}_0^k(u): u \text{ is not }(\eta,s,k+1)\text{-homogeneous at }x\text{ for any }s \in (0,1]\}.$$
\end{definition}

Standard arguments now imply the following.

\begin{proposition}\label{prop:Skequiv}
 $\mathcal{S}^k_0(u)=\bigcup_{\eta>0}\mathcal{S}^k_{0,\eta}(u)=\bigcup_{\eta>0}\bigcap_{r>0}\mathcal{S}^k_{0,\eta,r}(u).$   
\end{proposition}

\section{Local and Global Results}\label{sec:results}

\subsection{Local Results}
In this subsection, we will state our main results for maps into conical complexes (Theorems \ref{thm:k-minsk-local} and \ref{thm:k-rect-local}). The subsequent subsection contains the global results, Theorems \ref{thm:k-minsk} and \ref{thm:k-rect-quant}, which can be recovered from the ``local" results for conical $F$-connected complexes by standard covering arguments.

In the following, $B_r(U)$ denotes the $r$-tubular neighborhood about $U$.

\begin{theorem}\label{thm:k-minsk-local}
Suppose that $u$ is a harmonic map from $B_{64}(0)\subset\R^n$ into a conical $F$-connected complex $X_C$, satisfying \eqref{eq:Icond}.

Then there are constants $\eta_0(n),C=C(n,\eta,X_C,\Lambda)>0$ so that whenever $0<\eta<\eta_0(n)$, we have the Minkowski-type estimate
\[
|B_r(\mathcal{S}^k_{0,\eta}(u)\cap B_{1/8}(0))|\leq Cr^{n-k}.
\] 
\end{theorem}

\begin{theorem}\label{thm:k-rect-local}
Suppose that $u$ is a harmonic map from $B_{32}(0)\subset\R^n$ into a conical $F$-connected complex $X_C$ satisfying \eqref{eq:Icond}.

There exists an $\eta_0(n)>0$ so that if $0<\eta<\eta_0(n)$, then $\mathcal{S}^k_{0,\eta}(u)\cap B_{1}(0)$ is countably $k$-rectifiable.
\end{theorem}

\subsection{Global Results}Using the results of the previous section, we establish Minkowski bounds for the set $\mathcal S^k_{0,\eta}(u)$ and $k$-rectifiability for $\mathcal S^k_0(u)$.

\begin{theorem}\label{thm:k-minsk}
Let $u:\Omega\to X$ be a harmonic map, where $\Omega\subset\R^n$ is open and $X$ is an $F$-connected complex.  There exists an $\eta_0(n)>0$ so that for any compact $K\subset\Omega$ and any $0<\eta<\eta_0(n)$, we have $\mathcal{H}^{k}(\mathcal{S}^k_{0,\eta}(u)\cap K)<\infty$ and indeed for $r<1$,
\begin{equation}
|B_r(\mathcal{S}^k_{0,\eta}(u)\cap K)|\leq C(K,u,\eta)r^{n-k}.
\end{equation}
\end{theorem}

\begin{theorem}\label{thm:k-rect-quant}
Let $u:\Omega\to X$ be a harmonic map, where $\Omega\subset\R^n$ is open and $X$ is an $F$-connected complex. There exists an $\eta_0(n)>0$ so that for every $0<\eta<\eta_0(n)$ the set $\mathcal{S}^k_{0,\eta}(u)$ is countably $k$-rectifiable.  In particular, by Proposition \ref{prop:Skequiv}, $\mathcal{S}^k_0(u)$ is countably $k$-rectifiable as well.
\end{theorem}
To recover Theorems \ref{thm:k-rect-quant} and \ref{thm:k-minsk} from the conical results of Theorem \ref{thm:k-rect-local} and \ref{thm:k-minsk-local}, one can use the local conicality of $F$-connected complexes and a standard covering argument. We do not carry out those arguments here. 

\subsection{Proof of Rectifiability of the Singular Strata}\label{sec:proof of main}

Finally, we prove Theorem \ref{thm:k-rect} from Theorem \ref{thm:k-rect-quant}.

\begin{proof}[Proof of Theorem \ref{thm:k-rect}]
Writing $\mathcal{S}^k(u)=\bigcup_{j=0}^{N-1}\mathcal{S}^k_j(u)$, it clearly suffices to show that each $\mathcal{S}^k_j(u)$ is $k$-rectifiable.  First, we observe that $\mathcal{S}^k_0(u)$ is $k$-rectifiable by Theorem \ref{thm:k-rect-quant}.

Now fix $j \in \{1, \dots, n-2\}$ and let $x \in \mathcal S^k_j(u)$. Recalling \eqref{eq:sigmaj}, consider splitting data $(\sigma_j(x),u_1^x,u_2^x,j,Y_x)$ for $u$ at $x$.  It follows directly from the definitions that $\mathcal{S}^k_{j}(u)\cap B_{\sigma_j(x)}(x)\subseteq\mathcal{S}^k_0(u_2^x)$, because for any $y\in\mathcal{S}^k_j(u)\cap B_{\sigma_j(x)}(x)$ the splitting data $(\sigma_j(x)-d(x,y),u_1^x,u_2^x,j,Y_x)$ has optimal dimension and hence $\sigma_j(x)-d(x,y)\leq\sigma_j(y)$.  Covering $\mathcal{S}^k_j(u)$ by countably many balls $B_{\sigma_j(x)}(x)$, since each set $\mathcal{S}^k_0(u_2^x)$ is $k$-rectifiable by Theorem \ref{thm:k-rect-quant}, so is $\mathcal{S}^k_j(u)$. It follows that $\mathcal{S}^k(u)$ is $k$-rectifiable.
\end{proof}

\section{Cone Splitting}\label{sec:conesplitting}

The proofs of Theorems \ref{thm:k-minsk-local} and \ref{thm:k-rect-local} can be deduced using standard methods, after establishing a number of key structural lemmas. We collect these structural lemmas in this section. In a number of these results, we need a notion of ``spanning a $k$-dimensional affine subspace" which is stable under limits. The following definition is commonly used e.g. in \cite{nv17,dmsv,dees}. 

\begin{definition}
    Let $\{x_0, \dots x_k\} \subset B_1(0) \subset \mathbb R^n$. We say that this set {\bf $\rho$-spans} or {\bf $\rho$-effectively spans} a $k$-dimensional affine subspace if for all $i=1, \dots, k,$
    \[
x_i \notin B_{\rho}(x_0+ \text{span}(x_1-x_0, \dots x_{i-1}-x_0)).
    \]
    
    Given a set $F \subset B_1(0)$ we say that $F$ {\bf $\rho$-spans} or {\bf $\rho$-effectively spans} a $k$-dimensional affine subspace if there exist a set $\{x_0, \dots, x_k\} \subset F$ that $\rho$-effectively spans a $k$-dimensional affine subspace.
\end{definition}

In this work, the monotone quantity is the order or frequency, so it is useful to define the following quantity (cf. \cite{dmsv,dees}).

\begin{definition}
    For $0<s<r$, the {\bf order pinching} or {\bf frequency pinching} is defined by
\begin{equation}\label{eq:freqpinch}
    W_s^r(x):=\text{Ord}_\phi(x,r)-\text{Ord}_\phi(x,s).
\end{equation}
When multiple maps are under consideration, we will sometimes write $W_s^r(x,u)$ to disambiguate.
\end{definition}

\subsection{Known Results}

In the proofs of this section, we will sometimes use the following four results, the proofs of which appear in \cite{dees}.

\begin{lemma}[Lemma 4.2, \cite{dees}]\label{lem:ben4.2}
    There is a constant $C$ so that for all $y\in B_{r/4}(x)$,
    \[
    \Ord_\phi(y,r)\leq C(\Ord_\phi(x,16r)+1).
    \]
\end{lemma}

\begin{lemma}[Lemma 5.1, \cite{dees}]\label{lem:ben5.1}
    If $u$ is a harmonic map $u:B_r(0)\to X_C$, where $X_C$ is a conical $F$-connected complex, and $\Ord_\phi(0,s)=\Ord_\phi(0,r)$ for $s<r$, then $u$ is homogeneous about $0$ on $B_{r}(0)$.
\end{lemma}

\begin{lemma}[Lemma 6.3, \cite{dees}]\label{lem:ben6.3}
    Suppose that $u_k:B_{4}(0)\to X_C$ are harmonic maps into a conical $F$-connected complex $X_C$ with $0_X\in u_k(B_2(0))$ and so that \begin{enumerate} \item $\Ord_{\phi,u_k}(0,4)$ and $I_{\phi,u_k}(0,4)$ are uniformly bounded (in $k$), and \item $I_{\phi,u_k}(0,4)$ is uniformly bounded away from $0$.
\end{enumerate}

Then there is a subsequence $u_\ell$ of the $u_k$ converging uniformly in $B_{2}(0)$ and strongly in $W^{1,2}(B_{2}(0))$ to a nonconstant harmonic map $u:B_{2}(0)\to X$.
\end{lemma} 

\begin{lemma}[Lemma 7.4, \cite{dees}]\label{lem:ben7.4}
    Suppose that $\Omega\subset\R^n$ is convex open domain.  Suppose also that $u:\Omega\to X_C$ is a nonconstant continuous map into a conical $F$-connected complex which is homogeneous with respect to two distinct points $x_1$ and $x_2$, with orders $\alpha_1$ and $\alpha_2$, respectively.

Then $\alpha_1=\alpha_2$ and $u$ is independent of the $x_1-x_2$ direction, that is, \[
u(y+\lambda(x_1-x_2))=u(y)\] for any $y\in\R^n,\lambda\in\R$ for which both $y$ and $y+\lambda(x_1-x_2)$ lie in $\Omega$.
\end{lemma}

\subsection{$k$-Homogeneity and Pinching}

The following lemmas give us tools to show that in certain balls, the singular set lies close to an affine $k$-plane.

\begin{lemma}\label{lem:highordhom}
Let $u:B_r(0)\to X_C$ be a harmonic map into a conical $F$-connected complex, and suppose that $u$ is $k$-homogeneous with respect to the $k$-plane $V$.  Suppose that $y\notin V$ and $u(y)=0_X$.  Then any tangent map of $u$ at $y$ is $(k+1)$-homogeneous.

In particular, for any $\eta>0$ there is some $r_\eta>0$ so that $u$ is $(\eta,r_\eta,k+1)$-homogeneous at $y$.
\end{lemma}

\begin{proof}
Consider any tangent map $u_*=\lim_{n\to\infty}u_{r_n}$ of $u$ at $y$, where $u_{r_n}(z)=u(y+r_nz)$ as in \eqref{eq:rescalings}, so that $u_*$ is realized as in Definition \ref{def:tangentmap}.  Then, we can use the scale-invariance of the order to compute that
\[
\text{Ord}_{u_*}(y,\sigma)=\lim_{n\to\infty}\text{Ord}_u(y+r_ny,r_n\sigma).
\]
Moreover, by the homogeneity of $u$ with respect to $0$, we have that
\[
\lim_{n\to\infty}\Ord_u(y+r_ny,r_n\sigma)=\lim_{n\to\infty}\Ord_u\left(y,\frac{r_n}{1+r_n}\sigma\right)=\Ord_u(y).
\]
Since this last expression is independent of $\sigma$, we conclude that $\text{Ord}_{u_*}(y,\sigma)$ is independent of $\sigma$ and hence by Lemma \ref{lem:ben5.1}, $u_*$ is homogeneous about $y$.

Since $u_*$ is homogeneous about $0$ and about $y$, by Lemma \ref{lem:ben7.4}, $u_*$ is independent of the span of $y$.  Also, because $u$ is invariant with respect to $V$, 
$u_*:B_1\to X_{u(0)}$ is itself invariant with respect to $V$.  In particular, $u_*$ is invariant with respect to the $(k+1)$-dimensional subspace spanned by $V$ and $y$. Therefore $u_*$ is $(k+1)$-homogeneous about $y$.
\end{proof}

\begin{lemma}\label{lem:cone-split}
 Let $u:B_4(0) \to X_C$ be a harmonic map to a conical $F$-connected complex satisfying \eqref{eq:Icond} such that $0_X \in u(B_2(0))$. Then for all $\eta,\rho>0$, there exists a $\delta=\delta(n,\rho, \eta, X_C, \Lambda)>0$ such that if $F=\{y \in B_1(0) : W_\rho^2(y)<\delta\}$ $\rho$-effectively spans a $k$-dimensional affine subspace $V$ then
    \[
B_2(0)\cap \mathcal{S}^k_{0,\eta,\delta}(u)
\subset B_{2\rho}(V).
    \]
\end{lemma}

\begin{proof}
    We proceed by contradiction. Assume that no such $\delta$ exists. Then there exists a sequence of harmonic maps $u_i:B_4(0) \to X$ satisfying \eqref{eq:Icond} so that:
    \begin{itemize}
        \item There exist collections $\{x_0^i, \dots, x_k^i\}\subset B_1(0)$ which $\rho$-effectively span an affine $k$-plane $V^i$.
        \item For each $j\in \{0, \dots, k\}$, $W_{\rho}^2(x_j^i,u_i)<\frac 1i$.
        \item There exist $y^i\in\mathcal{S}^k_{0,\eta,1/i}(u_i) \setminus B_{2\rho}(V^i)$ (in particular, $u_i(y^i)=0_X$).
    \end{itemize}
    We may regularize the maps so that $I_{\phi,u_i}(0,4)=1$, to apply Lemma \ref{lem:ben6.3}.

    By Lemma \ref{lem:ben6.3}, there is a(n unrelabeled) subsequence of maps $u_i:B_2(0)\to X_C$ such that $u_i$ converges uniformly to a nonconstant harmonic map $u$, the collection $\{x_0^i, \dots x_k^i\}$ converges to a collection $\{x_0, \dots, x_k\}\subset \overline{B_1(0)}$ which $\rho$-effectively spans some $k$-plane $V$ and such that $W_{\rho}^2(x_j,u)=0$ for all $j=0, \dots, k$, and $y^i$ converges to some point $y\notin V$ such that $u(y)=0_X$. By Lemma \ref{lem:ben5.1}, $u$ is homogeneous of some degree $\alpha_j\geq 1$ on $B_1(x_j)$ for each $j=0, \dots, k$. Moreover, since $\overline{B_1(0)} \subset \text{cvx}\left(\cap_{j=0}^k B_2(x_j)\right)$, by Lemma \ref{lem:ben7.4}, all of the $\alpha_j$ are equal and $u$ is homogeneous of degree $\alpha$ about each of the $x_j$, and is hence $k$-homogeneous about each of them.

    Now, because $u(y)=0_X$, and $u$ is homogeneous about the $k$-plane $V$, by Lemma \ref{lem:highordhom}, any tangent map $u_*$ at $y$ is $(k+1)$-homogeneous. In particular, there is some $r_\eta$ so that $u$ is $(\frac\eta4,r_\eta,k+1)$-homogeneous at $y$.

\begin{claim}
    For every $(k+1)$-homogeneous map $h$,
       \[
    \sup_{B_r(y)}d(u,h^y)\geq\frac{\eta}{2}\left(\frac{I_{\phi,u}(y,r)}{r^{n-1}}\right)^{\frac12}
    \]
    where $r:=r_\eta$ for convenience.
\end{claim}
\begin{proof}Suppose to the contrary that for some $r>0$ there is a $(k+1)$-homogeneous $h$ with
    \[  
    \sup_{B_r(y)}d(u,h^y)<\frac{\eta}{2}\left(\frac{I_{\phi,u}(y,r)}{r^{n-1}}\right)^{\frac12}.
    \]Since the $u_i\to u$ uniformly, for each fixed $r>0$, $I_{\phi,u_i}(y^i,r)\to I_{\phi,u}(y,r)$, so that in particular for $i$ sufficiently large,
    \[
    \frac23\left(\frac{I_{\phi,u}(y,r)}{r^{n-1}}\right)^{\frac12}\leq\left(\frac{I_{\phi,u_i}(y^i,r)}{r^{n-1}}\right)^{\frac12}\leq\frac32\left(\frac{I_{\phi,u}(y,r)}{r^{n-1}}\right)^{\frac12}.
    \]

    Since $h^{y^i}\to h^y$ uniformly on $B_{2r}(y)$ as $i\to\infty$ (because $y^i\to y$), for sufficiently large $i$ we have that
    \[
    \sup_{B_r(y^i)}d(u_i,h^{y^i})\leq\frac{7\eta}{8}\left(\frac{I_{\phi,u_i}(y^i,r)}{r^{n-1}}\right)^{\frac12}
    \]
    which, whenever $\frac1i<r$, contradicts the fact that $y^i\in\mathcal{S}^k_{0,\eta,1/i}(u_i)$.
\end{proof}

However, the above Claim contradicts the statement that $u$ is $(\frac\eta4,r_\eta,k+1)$-homogeneous at $y$.
    \end{proof}

\begin{lemma}\label{lem:ord-pinch}
    Let $u:B_{64}(0) \to X_C$ be a harmonic map to a conical $F$-connected complex satisfying \eqref{eq:Icond} and with $0_X \in u(B_2(0))$. Let $\rho, \epsilon >0$ and presume that $\sup_{B_2(0)}\text{Ord}_\phi(y,2) \leq \Gamma$. Then there exists $\delta=\delta(n,\rho,\epsilon,X_C, \Lambda)>0$ such that if $F=\{y \in B_{1}(0): \text{Ord}_\phi(y,\rho)>\Gamma-\delta\}$ $\rho$-effectively spans a $k$-dimensional affine subspace $V$ then for all $x \in V \cap B_{1}(0)$, 
    \[
\text{Ord}_\phi(x,2\rho^2)> \Gamma - \epsilon.    
\]

\end{lemma}

    \begin{proof}

   Presume no such $\delta>0$ exists. Then there exists a sequence of such maps $u_i$ with corresponding sets $F_i$ which $\rho$-effectively span $V_i$ using the points $\{y_{0}^i, \dots, y_{k}^i\}\subset B_{1}(0)$. Moreover $\text{Ord}_{\phi, u_i}(y_{j}^i,\rho)>\Gamma_i - \frac1i$, but there exists a point $x^i \in V_i \cap B_{1}(0)$ such that $\text{Ord}_{\phi,u_i}(x^i,2\rho^2) \leq \Gamma_i - \epsilon$. We note that by Lemma \ref{lem:ben4.2}, we have the uniform bound $\Gamma_i\leq C(\Lambda)$. Then normalizing $I_{\phi,u_i}(0,4)=1$ and using Lemma \ref{lem:ben6.3} we can find a subsequence such that:
   \begin{itemize}
       \item $u_i \to u$ uniformly in $C^0(B_2(0))$ and strongly in $W^{1,2}(B_2(0))$, where $u$ is a harmonic map.
       \item $\Gamma_i\to\Gamma$.
       \item $V_i \to V$, a $k$-dimensional affine subspace and $V$ is $\rho$-effectively spanned by $\{y_0,\dots, y_k\}\subset \overline{B_{1}(0)}$, where $y_j= \lim_{i \to \infty} y_j^i$, $\text{Ord}_{\phi,u}(y_j, \rho) \geq \Gamma$.
       \item Finally $x^i \to x\in V\cap \overline{B_{1}(0)}$ such that $\text{Ord}_{\phi,u}(x,2\rho^2) \leq \Gamma - \epsilon$.
   \end{itemize}
Since $W_{\rho}^2(y_j,u)=0$, the map $u$ is homogeneous of some degree $\alpha_j\geq 1$ on $B_2(y_j)$. Then, as in the proof of Lemma \ref{lem:cone-split}, we use Lemmas \ref{lem:ben5.1} and \ref{lem:ben7.4} to see that $u$ is homogeneous of a fixed degree $\alpha=\Gamma$ about each $y_j$, and is invariant with respect to the span of the $y_j$.  In particular, $u$ must be homogeneous of degree $\Gamma$ about every point in $V\cap B_{1}(0)$ and is invariant with respect to $V$. By the constancy of $\Ord_{\phi,u}(\cdot,2\rho^2)$ in $V$, we must have that $\Ord_{\phi,u}(x,2\rho^2)=\Gamma$, which is a contradiction.
\end{proof}

\subsection{Translating Inhomogeneities}\label{sec:conesplittingtranslating}

In the key covering arguments, an important note is that if a point $y$ is pinched, and is not $(\eta,r,k+1)$-homogeneous, then pinched points close to $y$ also cannot be $(\eta,r,k+1)$-homogeneous.  This subsection provides two results that make this remark precise, the proofs of which differ depending on the ``closeness" of $y$ to the comparison point. 

Before proving these results, we prove two lemmas that give us more precise control on the gradient of a homogeneous map approximating a harmonic map.

\begin{lemma}\label{lem:hom-grad-bound}
        Suppose that $u:B_8(0)\to X_C$ is a harmonic map satisfying \eqref{eq:Icond} and $0_X\in u(B_1(0))$. Let $h:\mathbb R^n \to X_C$ be a $k$-homogeneous map of order $\alpha$ such that 
        \[
        \sup_{B_1(0)}d(u,h) \leq C'.
        \]Then there exists a $k$-homogeneous map $\widehat h:\mathbb R^n \to X_C$ satisfying
        \[
         \sup_{B_1(0)}d(u,\widehat h) \leq 2C'
        \]and such that 
        \[
        \sup_{B_1(0)} |\nabla\widehat h| \leq  (1+4\alpha)\sup_{ B_2(0)} |\nabla u|.
        \]
\end{lemma}

\begin{proof}
Presume that $h$ is invariant with respect to $V=\R^k \times \{0\}$. For $(v,w) \in \R^k \times \R^{n-k}$ define
\[
\widehat h(v,w) =\left\{\begin{array}{ll}
|w|^\alpha u\left(0,\frac w{|w|}\right)& \text{ if } w \neq 0,\\
0_X & \text{ otherwise}.
\end{array}\right.
\]Observe that $\widehat h$ is invariant with respect to $V$ and, since (for $w \neq 0$, $\lambda>0$)
\[
\widehat h(\lambda v, \lambda w)=\lambda^\alpha |w|^\alpha u\left(0,\frac w{|w|}\right)=\lambda^\alpha \widehat h(v,w),
\]$\widehat h$ is homogeneous of degree $\alpha$.

 Now for all $(v,w) \in \Sp^{n-1}$, if $w \neq 0$ then
 \[
 d(\widehat h(v,w), h(v,w))= d\left( |w|^\alpha u\left(0,\frac w{|w|}\right), |w|^\alpha h\left(0,\frac w{|w|}\right)\right) \leq C'.
 \]And if $w=0$ then $\widehat h(v,0)=h(v,0)=0_X$. By properties of homogeneity, it follows that 
 \[
 \max_{B_1(0)}d(u, \widehat h) \leq  \max_{B_1(0)}d(u,h) +  \max_{B_1(0)}d(h, \widehat h) \leq 2C'.
 \]The gradient bound is immediate.
    \end{proof}

\begin{lemma}\label{lem:hom-ord-bound}
    Suppose that $u:B_8(0)\to X_C$ is a harmonic map satisfying \eqref{eq:Icond} with $I(0,1)=1$ and $0_X\in u(B_1(0))$.  There exist $\eta_0(n)>0$ and $A(\Lambda,n)>1$ so that for any $0<\eta<\eta_0$, if $h:B_8(0)\to X_C$ is a homogeneous map such that
    \[
    \sup_{B_1(0)}d(u,h)\leq\eta,
    \]
    then $\alpha\leq A$, where $\alpha$ is the degree of $h$.
\end{lemma}

\begin{proof}
    In this proof, whenever $\Ord_\phi,E_\phi,I_\phi$ appear, they are the quantities for the map $u$, not for the map $h$ (the latter map's energy, height, and order will not be relevant in the proof).   We see directly that $\Ord_\phi(0,1)=E_\phi(0,1)$.  We also have a uniform upper bound on $I_\phi(0,8)$ from \eqref{eq:height-technical} and hence an upper bound on $E_\phi(0,8)\geq E(0,4)$.  From this energy bound, we deduce $\sup_{B_2(0)}|\nabla u| \leq C(\Lambda)$ by \cite[Theorem 2.4.6]{korevaar-schoen1}. 

    For $0<r\leq1$, if $\max_{\partial B_r(0)} d(u,0_X) \leq \tau$ then $I_\phi(0,r)\leq C(n)\tau^2$.  Fix $\eta_0(n)$ so that $\beta(n):=C(n)(2\eta_0)^2<1$ and consider any $\eta\in(0,\eta_0)$. By \eqref{eq:energy-deriv} and the uniform gradient bound for $u$, we can bound the derivative $\partial_rE_\phi(0,r)\leq C(n,\Lambda)$ for $r \in (0,2)$. Using this derivative bound, choose $r_0(n,\Lambda)<1$ so that 
    \[
    r_0E_\phi(0,r_0)\geq  \beta E_\phi(0,1).
    \]

By hypothesis,
$\max_{\partial B_1(0)} d(h,0_X)\leq C(\Lambda)+\eta\leq C(n,\Lambda)$.  In particular, by homogeneity we have that
    \[
\max_{\overline{B_{r_0}(0)}}    d(h,0_X)\leq r_0^\alpha C(n,\Lambda).
    \]
Choose $A(n,\Lambda)$ such that $r_0^AC(n,\Lambda)<\eta_0$.
If $\alpha\geq A$, then $\sup_{B_{r_0}(0)}d(h,0_X)\leq \eta_0$ so $\sup_{B_{r_0}(0)}d(u,0_X)\leq(\eta+\eta_0)$ and \[I_\phi(0,r_0)\leq C(n)(\eta+\eta_0)^2<C(n)(2\eta_0)^2=\beta.\]  
 We compute $\Ord_\phi(0,r_0)$ directly:
    \[
    \Ord_\phi(0,r_0)=\frac{r_0E_\phi(0,r_0)}{I_\phi(0,r_0)}>\frac{\beta E_\phi(0,1)}{\beta}=E_\phi(0,1)=\Ord_\phi(0,1).
    \]
    This contradicts the monotonicity of the order and therefore $\alpha<A$.
\end{proof}

Combining Lemma \ref{lem:hom-grad-bound}, Lemma \ref{lem:hom-ord-bound}, and Lipschitz bounds for $u$ we see that if $u:B_8(0)\to X_C$ is a harmonic map with $0_X\in u(B_1(0))$ satisfying \eqref{eq:Icond} and $I_\phi(0,1)=1$, and $h$ is a $k$-homogeneous map so that
\[
\sup_{B_1(0)}d(u,h)\leq\eta,
\]
where $\eta<\eta_0(n)$ is sufficiently small, then there is a $k$-homogeneous map $\widehat{h}$ so that
\[
\sup_{B_1(0)}d(u,\widehat{h})\leq2\eta
\]
and
\[
\sup_{B_1(0)}|\nabla\widehat{h}|\leq C(n,\Lambda).
\]

\begin{lemma}\label{lem:move-inhom}
      Let $u:B_{64}(0) \to X_C$ be a harmonic map into a conical $F$-connected complex with $0_X \in u(B_1(0))$ satisfying \eqref{eq:Icond}. Given $0<\rho<1$ and $0<\eta<\eta_0(n)$, where $\eta_0(n)$ is chosen as in Lemma \ref{lem:hom-ord-bound}, there exists $\delta=\delta(n,\rho, \eta,X_C,\Lambda)>0$ such that if 
      \begin{itemize}
          \item the set
          \[
          F(B_{1}(0)):= \{z \in\mathcal S^k_{0, \eta, \delta}(u) \cap B_1(0):W^{8}_1(z) < \delta\}
          \]$\rho$-effectively spans a $k$-dimensional affine space $V$,
          \item $x \in B_1(0) \cap V$ with $W^{8}_1(x)<\delta$,
      \end{itemize}then $u$ is not $(\eta/3,t,k+1)$-homogeneous at $x$ for all $t \in [2\rho,2]$.
\end{lemma}

\begin{proof}
    Presume no such $\delta$ exists.  There is then a sequence of maps $u_i:B_{64}\to X_C$ satisfying \eqref{eq:Icond}, with corresponding sets $\{y_0^i,\dots,y_k^i\}\subset B_1(0)$ which $\rho$-effectively span affine subspaces $V^i$.  Moreover, these $y_j^i$ are such that $W_1^8(y_j^i)<\frac1i$, and $y_j^i\in\mathcal{S}^k_{0,\eta,1/i}(u_i)$.  Finally, there is a sequence of points $x^i\in B_1(0)\cap V^i$ with $W_1^8(x^i)<\frac1i$ and a sequence of scales $t_i\in[2\rho,2]$ so that $u$ is $(\eta/3,t_i,k+1)$-homogeneous at $x_i$, so that in particular there is a $(k+1)$-homogeneous map $\widetilde{h}_i$ so that
    \[
    \sup_{B_{t_i}(x_i)}d(u_i,\widetilde{h}_i^{x_i})\leq\frac\eta3\left(\frac{I_{\phi,u_i}(0,t_i)}{t_i^{n-1}}\right)^{\frac{1}{2}}.
    \]

    We apply Lemma \ref{lem:hom-grad-bound} to obtain a new sequence of $(k+1)$-homogeneous maps $h_i$ with
    \[
    \sup_{B_{t_i}(x_i)}d(u_i,{h}_i^{x_i})\leq\frac{2\eta}3\left(\frac{I_{\phi,u_i}(0,t_i)}{t_i^{n-1}}\right)^{\frac{1}{2}}
    \]
    and
    \[
    \sup_{B_{t_i}(x_i)}|\nabla h_i^{x_i}|\leq C(n,\Lambda)
    \]
    where we have used the computations following Lemma \ref{lem:hom-ord-bound} to bound $|\nabla h_i|$ in terms of $n,\Lambda$.  Here, it is useful to recall that $I_\phi(x_i,t_i)$ is uniformly bounded above and below in terms of $\Lambda,\rho$, by \eqref{eq:height-technical}; this bounds the scaling factor in the rescalings we use to apply Lemma \ref{lem:hom-grad-bound}.  
    
    All of the $t_i\geq2\rho$ by assumption, and we can of course find a subsequence of the $x_i$ converging to some $x$.  By Arzela-Ascoli, the gradient bounds allow us to find a subsequence of the $h_i$ which converges uniformly on $B_{2\rho}(x)$ to some map $h$; due to this uniform convergence it is not hard to see that $h$ is $(k+1)$-homogeneous.  By homogeneity, this uniform convergence can in fact be extended to $B_3(x)$.

    Using the above and applying Lemma \ref{lem:ben6.3}, we can find a subsequence so that:
    \begin{itemize}
        \item $u_i \to u$ uniformly in $C^0(B_2(0))$ and strongly in $W^{1,2}(B_2(0))$, where $u$ is a harmonic map.
         \item $V^i \to V$, a $k$-dimensional affine subspace and $V$ is $\rho$-effectively spanned by $\{y_0,\dots, y_k\}\subset \overline{B_{1}(0)}$, where $y_j= \lim_{i \to \infty} y_j^i$, noting that $W_1^8(y_j)=0$ for each $j=0,\dots,k$.
         \item $x^i\to x\in V\cap\overline{B_1(0)}$, $W_1^8(x)=0$.
         \item $h_i\to h$ uniformly in $C^0(B_2(0))$, where $h$ is a $(k+1)$-homogeneous map.
         \item $t_i\to t\in[2\rho,2]$, and moreover
         \[
             \sup_{B_{t}(x)}d(u,h^{x})\leq\frac{2\eta}3\left(\frac{I_{u}(0,t)}{t^{n-1}}\right)^{\frac{1}{2}}.
         \]
    \end{itemize}

    We observe that since $u$ is homogeneous about all of the $y_j$, it is homogeneous about each of them and it is invariant with respect to $V$, their span.  In particular, since $x\in V$, we immediately see that for each $y_j$,
    \[
    \sup_{B_{t}(y_j)}d(u,h^{y_j})\leq\frac{2\eta}3\left(\frac{I_{u}(0,t)}{t^{n-1}}\right)^{\frac{1}{2}}.    
    \]
    by the invariance of $u$.  For large enough $i$, this contradicts that $y_j^i\in\mathcal{S}^k_{0,\eta,1/i}(u_i)$.
\end{proof}

\begin{lemma}\label{lem:move-inhom-large}
  Let $u:B_{8}(0) \to X_C$ be a harmonic map into a conical $F$-connected complex with $0_X \in u(B_1(0))$ satisfying \eqref{eq:Icond}. Given $0<\rho<2$, and $0<\eta<\eta_0(n)$, where $\eta_0(n)$ is chosen as in Lemma \ref{lem:hom-ord-bound}, there exists $\delta=\delta(n,\rho, \eta,X_C,\Lambda)>0$ such that if $W_{\rho}^{4}(0)<\delta$ and there exists $y \in B_{1}(0)$ such that
  \begin{itemize}
      \item $W_{\rho}^{4}(y)<\delta$
      \item $u$ is not $(\eta, 1, k+1)$-homogeneous about $y$.
  \end{itemize}Then $u$ is not $(\eta/3, 1, k+1)$-homogeneous about $0$.
\end{lemma}

\begin{proof}
        As usual, we proceed by contradiction. Suppose no such $\delta>0$ exists. Then there exists a sequence $u_i:B_8(0) \to X_C$ of harmonic maps with $I_{\phi,u_i}(0,8)=1$ such that
    \begin{itemize}
        \item $W_{\rho}^4(0,u_i)<\frac 1i$
        \item there exists a sequence $y_i \in B_{1}(0)$ such that $W_{\rho}^4(y_i,u_i)<\frac 1i$
        \item $u_i$ is not $(\eta, 1, k+1)$-homogeneous about $y_i$
        \item $u_i$ is $(\eta/3, 1, k+1)$-homogeneous about $0$ with comparison map $\widetilde{h}_i$.
    \end{itemize}
    The sequence $u_i \to u$ in $C^0\cap W^{1,2}(B_4(0))$ where $u$ is a nonconstant harmonic map, $y_i \to y \in \overline{B_{1}(0)}$, and $W_{\rho}^4(0,u)=W_{\rho}^4(y,u)=0$. As in the proof of Lemma \ref{lem:move-inhom}, we can use Lemmas \ref{lem:hom-grad-bound} \ref{lem:hom-ord-bound} to replace the sequence $\widetilde{h}_i$ with a sequence $h_i$ converging (up to a subsequence) to $h$, a $(k+1)$-homogeneous map with
    \begin{equation}\label{eq:sym-0}
    \sup_{B_{1}(0)}d(u,h)\leq\frac{2\eta}3\left(I_{\phi,u}(0,1)\right)^{\frac12}.
    \end{equation}

    By Lemma \ref{lem:ben5.1}, $u$ is homogeneous about both $0$ and $y$ and in particular if $y\neq0$, by Lemma \ref{lem:ben7.4} is invariant along the line between them.  Hence on $B_{1}(y)$, by \eqref{eq:sym-0} we have
    \begin{equation}\label{eq:move-inhom-y-symm}
    \sup_{B_{1}(y)}d(u,h^y)\leq \frac{2\eta}3\left(I_{\phi,u}(0,1)\right)^{\frac12}
    \end{equation}
    where $h^y(z):=h(z-y)$ as in Definition \ref{def:quant-hom}.  However, because $u_i\to u$ uniformly on $B_{3}(y)$, and $y_i\to y$, for large enough $i$ this contradicts the fact that $u_i$ is not $(\eta,1,k+1)$-homogeneous about $y_i$.
    
    On the other hand, if $y=0$, we immediately obtain \eqref{eq:move-inhom-y-symm} from \eqref{eq:sym-0}, which is still a contradiction.
    
\end{proof}

\section{The Mean Flatness}\label{sec:meanflat}

Thus far, everything we have seen involves the frequency or the frequency pinching; it all relates to the monotone quantity for harmonic maps into $F$-connected complexes. 
 However, the tools we will use to prove the main theorems require bounds on a different quantity, namely the mean flatness of certain measures, sometimes called the Jones' $\beta_2$ numbers.

\begin{definition}
Let $\mu$ be a Radon measure on $\R^n$ and $k\in\{0,1,\dots,n-1\}$.  For $x\in\R^n$ and $r>0$, we define the {\bf $k^{\text{th}}$ mean flatness} of $\mu$ in the ball $B_r(x)$ to be
\[
D_\mu^k(x,r):=\inf_Lr^{-k-2}\int_{B_r(x)}\text{dist}(y,L)^2\phantom{i}d\mu(y),
\]
where the infimum is taken over all affine $k$-planes $L$.
\end{definition}

In this section, we prove that, as long as $u$ is not $(\eta,r,k+1)$-homogeneous at $x$, bounding the frequency pinching in $B_r(x)$ allows us to bound the $k^{\text{th}}$ mean flatness.  In particular, this lets us bound the mean flatness of measures supported near the set $\mathcal{S}^k_{0,\eta,r}$.

\begin{lemma}\label{lem:mean-flat}
Suppose that $u:B_{4}(0)\subset\R^n \to X_C$ is a harmonic map satisfying \eqref{eq:Icond} and $X_C$ is a conical $F$-connected  complex. Assume further that $0_X \in u(B_2(0))$. For any $\eta>0$, $\rho\in(0,1)$, there are constants $C=C(n, \rho, \eta, X_C, \Lambda),\delta=\delta(n, \rho, \eta, X_C, \Lambda)>0$ for which the following holds:  

Suppose that $\mu$ is a finite nonnegative Radon measure and $u$ is not $(\eta, r,k+1)$-homogeneous at $x \in B_1(0)$ where $r \in (0,1)$.  If $W_{\rho r}^{2r}(x)<\delta$, then
\[
D_\mu^{k}(x,r/8)\leq\frac{C}{r^{k}}\int_{B_{r/8}(x)}W_{r/8}^{4r}(y)\phantom{i}d\mu(y).
\]
\end{lemma}

To prove this, we will need two auxiliary lemmas.

\begin{lemma}\label{lemma:Unique Extension}
Let $x_1,\dots,x_n$ be orthonormal coordinates for $\R^n$ and let $u:\Omega\to X$ be a harmonic map from a convex domain into an $F$-connected complex.  Suppose that $B_r(x_0) \subset \Omega$ and $u|B_r(x_0)$ is a function depending only on $x_1,\dots,x_j$ for $j<n$. Then $u$ is a function of only $x_1,\dots,x_j$ on all of $\Omega$.
\end{lemma}

This lemma is reminiscent of the classical unique continuation of harmonic maps, albeit in a very weak sense.  We do not outline the proof here as it follows nearly identically the proof of \cite[Lemma 6.2]{dees}, making only minor modifications to account for the fact that $u$ is now permitted to depend on $j$ variables.  We primarily need this result to prove our second auxiliary lemma.

\begin{lemma}\label{lemma:Flat Directions}
Suppose that $u:B_{4}(0)\subset\R^n \to X_C$ is a harmonic map satisfying \eqref{eq:Icond}, $X_C$ is a conical $F$-connected  complex, and $0_X \in u(B_2(0))$.  Then for any $\eta>0$, $\rho\in(0,1)$, there is a constant $\delta=\delta(n,\rho, \eta, X_C, \Lambda)>0$ so that the following holds. 

Suppose that $B_{r}(x)$ is a ball such that $u$ is not $(\eta, r,k+1)$-homogeneous at $x$, and $W_{\rho r}^{2r}(x)<\delta$.  Then for any orthonormal set $\{v_1,\dots,v_{k+1}\}$,
\[
r^{2-n}\int_{B_{5r/4}(x)\setminus B_{3r/4}(x)}\sum_{j=1}^{k+1}|\partial_{v_j}u|^2\geq\delta .
\]
\end{lemma}

\begin{proof}
We take $x=0$ by translation and $r=1$ by scale-invariance.  Fix $\eta>0$ and $\rho\in(0,1)$, and suppose that the claimed inequality does not hold for any $\delta>0$.

Then, for each $i$, there is a map $u_i$ with $I_{\phi,u_i}(0,1)=1$ which is not $(\eta,1,k+1)$-homogeneous at $x$, and so that $W_{\rho}^{2}(0,u_i)<\frac{1}{i}$.  Additionally, there exists an orthonormal set $\{v^i_{1},\dots,v^i_{k+1}\}$ so that
\[
\int_{B_{5/4}(0)\setminus B_{3/4}(0)}\sum_{j=1}^{k+1}|\partial_{v^i_{j}}u_i|^2<\frac{1}{i}.
\]

Applying Lemma \ref{lem:ben6.3}, we pass to a subsequence converging to some map $u$, which is homogeneous with respect to $0$ by Lemma \ref{lem:ben5.1}.   Taking convergent subsequences of the $v^i_j$, we find an orthonormal set $\{v_{1},\dots,v_{k+1}\}$ so that
\[
\int_{B_{5/4}(0)\setminus B_{3/4}(0)}\sum_{j=1}^{k+1}|\partial_{v_{j}}u|^2=0.
\]

This, of course, means that for any $B_\rho(y)\subset B_{5/4}(0)\setminus B_{3/4}(0)$, $u|_{B_\rho(y)}$ does not depend on the $v_1,\dots,v_{k+1}$ directions.  By Lemma \ref{lemma:Unique Extension}, $u$ does not depend on the $v_1,\dots,v_{k+1}$ directions anywhere on $B_2(0)$.

In particular, as a map homogeneous about $0$ which is invariant with respect to the subspace $\text{span}\{v_1,\dots,v_{k+1}\}$, $u$ is $(k+1)$-homogeneous at $0$.  This is in contradiction to the fact that the $u_i$ are not $(\eta,1,k+1)$-homogeneous at $0$, since the $u_i$ converge uniformly to $u$.
\end{proof}

The proof of Lemma \ref{lem:mean-flat} follows arguments similar to the proofs of \cite[Proposition 5.3]{dmsv} and \cite[Proposition 6.1]{dees}.  It is in this proof that we require the use of the smoothed order. The only alteration to the aforementioned proofs is that when we need to prove a bound of the form
\[
r^{2-n}\int_{B_{5r/4}(x)\setminus B_{3r/4}(x)}\sum_{i=1}^{k+1}|\partial_{v_i}u|^2\geq\delta 
\]
we invoke Lemma \ref{lemma:Flat Directions} above.  This is the one place that the hypothesis $u$ is not $(\eta,r,k+1)$-homogeneous at $x$ is used, and the one place where the hypothesis $W_{\rho r}^{2r}(x)<\delta$ is used.  In \cite{dmsv,dees} there is no need for these additional hypotheses because in those papers they are dealing with the ``top" $(n-2)$-stratum only---in their settings, a nonconstant harmonic map is {\em never} $(n-1)$ homogeneous, so the proof is simpler in those contexts.

\section{Proofs of Theorems \ref{thm:k-minsk-local} and \ref{thm:k-rect-local}}\label{sec:sketchy-proof}

The main theorems can now be proved by following the arguments of \cite{nv17} and \cite{dmsv} essentially identically. We state the following covering theorem and defer the proof to the appendix. The result follows as in the proofs of {\cite[Proposition 7.2]{dmsv}} and {\cite[Lemma 55]{nv18}} using our cone splitting results from Section \ref{sec:conesplitting}. 

\begin{theorem}\label{thm:k-2nd-cover}
    Let $u:B_8(0) \to X_C$ where $X_C$ is a conical $F$-connected complex and $u$ satisfies \eqref{eq:Icond}. 
    Presume $0_X \in u(B_2(0))$, $x\in B_{1}(0)$, and $0<s<S\leq\frac18$.  Given  $k \in \{0, \dots, n-1\}$ and $0<\eta<\eta_0(n)$, where $\eta_0(n)$ is chosen as in Lemma \ref{lem:hom-ord-bound}, there exist constants $\delta=\delta(n,\eta,X_C,\Lambda)>0$ and $C=C(n)>0$ so that the following holds. Suppose that $D\subseteq\mathcal{S}^k_{0,\eta,\delta s}\cap B_S(x)$ is any subset and set $M=\sup_{y\in D}\Ord_\phi(y,8S)$. 
    
    There is a finite decomposition of $D$ into sets $D_x$, a finite set $\mathcal C\subseteq B_2(0)$, and a collection of balls $\{B_{s_x}(x)\}_{x \in \mathcal C}$ with $s_x\geq s$, so that the following holds:
    \begin{enumerate}
        \item $D_x\subseteq B_{s_x}(x)$,
        \item $\sum s_x^k\leq CS^k$, and 
        \item for each $x\in \mathcal C$, either $s_x=s$ or for all $y\in D_x$,
        \begin{equation}\label{eq:ord-drop}
        \Ord_\phi(y,8s_x)\leq M-\delta.
        \end{equation}
    \end{enumerate}
\end{theorem}

With the above covering theorem, we prove our first main (local) theorem.

\begin{proof}[Proof of Theorem \ref{thm:k-minsk-local}]
 Choose $\eta_0(n)$ as in Lemma \ref{lem:hom-ord-bound}, so that we can invoke the covering theorem.  We consider the set $D_0=\mathcal{S}^k_{0,\eta}\cap B_{1/8}(0)$ and observe that due to Lemma \ref{lem:ben4.2}, we have that
    \[
    M:=\sup_{y\in D_0}\Ord_\phi(y,1)\leq C(\Lambda).
    \] Our aim is to establish a volume bound on $B_r(D_0)$.

    First, apply Theorem \ref{thm:k-2nd-cover} with $S=\frac 18$, $s=r$, $D=D_0$, and let $\{B_{s_x}(x)\}_{x\in \mathcal C_1}$ be the corresponding cover, with the corresponding $D_x\subseteq D_0$.  We have that
    \[
    \sum_{x \in \mathcal C_1}(s_x)^k\leq C.
    \]
    If $s_x=r$, we stop refining $B_{s_x}(x)$ and we place $x \in \mathcal C_1^g$, where $g$ denotes ``good". Let $\mathcal C_1^b:= \mathcal C_1 \setminus \mathcal C_1^g$ be the set of bad points.  Note that for the bad points, because $s_x>s$, we have that \[
    M_x:=\sup_{y\in D_x}\Ord_\phi(y,8s_x)\leq M-\delta.
    \]
    
    For each $x \in \mathcal C_1^b$, we again apply Theorem \ref{thm:k-2nd-cover} with $S=s_x$, $s=r$, and $D=D_x$ to obtain a covering $\{B_{s_y}(y)\}_{y \in \mathcal C_{2,x}}$ and a corresponding decomposition of $D_x$ into $D_{x,y}$ (which we relabel as $D_y$ at the next step).  
    
    For this application of Theorem \ref{thm:k-2nd-cover}, if $s_y>r$, we in fact have that
    \[
    \sup_{z\in D_{x,y}}\Ord_\phi(z,8s_y)\leq M_x-\delta\leq M-2\delta.
    \]
Let $\mathcal C_2:= \mathcal C_1^g \bigcup_{x \in\mathcal C_1^b} \mathcal C_{2,x}$. Then
    \[
    \sum_{x\in \mathcal C_2}(s_x)^k\leq C\sum_{x \in \mathcal C_1} (s_x)^k\leq C^2.
    \]
As before, define $\mathcal C_2^g:=\{x \in \mathcal C_2: s_x=r\}$ and let $\mathcal C_2^b:=\mathcal C_2 \backslash \mathcal C_2^g$. 

    We continue inductively so that at step $m$, there exists a covering $\{B_{s_x}(x)\}_{x\in \mathcal C_m}$ such that
    \[
    \sum_{x\in \mathcal C_m}(s_x)^k\leq C^m
    \]
    and for each $x\in \mathcal C_m$, either $s_x=r$ or
    \[
    \sup_{y\in D_x}\Ord_\phi(y,8s_x)\leq M-m\delta.
    \]
    If $m\geq\ceil{\frac M\delta}$, the second alternative cannot hold; this happens after a number of steps that is explicitly bounded in terms of $\Lambda,\delta$ and hence depends only on $n,\eta,X_C,\Lambda$.  At the least such $m$, we hence obtain a covering of $D$ by $\{B_r(x)\}_{x\in\mathcal C_m}$, so that
    \[
    \sum_{x \in \mathcal C_m}r^k\leq C^m.
    \]

    To obtain the desired Minkowski bound, we observe that because $D\subseteq\bigcup_{x\in\mathcal C_m}B_r(x)$, we have that $B_r(D)\subseteq\bigcup_{x\in\mathcal C_m}B_{2r}(x)$.  We bound the volume of this set
    \[
    |B_r(D)|\leq\sum_{x\in\mathcal C_m}\omega_n2^nr^n=\omega_n2^nr^{n-k}\sum_{x\in\mathcal C_m}r^k\leq\omega_n2^nC^mr^{n-k}=C'r^{n-k}
    \]
    where $C'=\omega_n2^nC^m$. It follows that, as desired, $C'$ depends only on $n,\eta,X_C,\Lambda$.
\end{proof}

To prove the rectifiability result, Theorem \ref{thm:k-rect-local}, we use the following result of \cite{at}.

\begin{theorem}[Appears as {\cite[Corollary 1.3]{at}}]\label{thm:at}
    Let $E\subseteq\R^n$ be an $\mathcal{H}^k$-measurable set with $\mathcal{H}^k(E)<\infty$, and let $\mu:=\mathcal{H}^k\llcorner E$.  The set $E$ is rectifiable if and only if
    \[
    \int_0^1D_\mu^k(x,s)\,\frac{ds}{s}<\infty
    \]
    for $\mu$-a.e. $x$.
\end{theorem}

\begin{proof}[Proof of Theorem \ref{thm:k-rect-local}]
Choose $\eta_0(n)$ as in Lemma \ref{lem:hom-ord-bound}, so that we can invoke Theorem \ref{thm:k-minsk-local}.  Set $\mu:=\mathcal{H}^k\llcorner\mathcal{S}^k_{0,\eta}(u)$. Suppose first that every pinching $W_s^t$ that appears is small enough to apply Lemma \ref{lem:mean-flat} whenever we need to apply it.  Then we compute for any $B_t(y)\subseteq B_1(0)$
\begin{align*}
\nonumber \int_{B_t(y)}\int_0^tD_\mu^k(z,s)\,\frac{ds}s\,d\mu(z)&\leq C\int_{B_t(y)}\int_0^ts^{-1-k}\int_{B_s(z)}W_s^{32s}(\xi)\,d\mu(\xi)\,ds\,d\mu(z)\\ 
&\leq C\int_0^t s^{-1-k}\int_{B_{s+t}(y)}W_s^{32s}(\xi)\int_{B_s(\xi)}d\mu(z)\,d\mu(\xi)\,ds\\
&\leq C'\int_0^ts^{-1}\int_{B_{s+t}(y)}W_s^{32s}(\xi)\,d\mu(\xi)\,ds\\
\nonumber &\leq C'\int_{B_{2t}(y)}\int_0^t{W_s^{32s}(\xi)}\,\frac{ds}s\,d\mu(\xi).
\end{align*}
The first inequality follows from applying Lemma \ref{lem:mean-flat}, the second from two applications of Fubini, and the third uses the fact that the Minkowski bound and a scaling argument imply that $\mu(B_s(\xi)) \leq C' s^k$.

The inner integral can be computed e.g. by dyadic decomposition cf. \cite{dmsv,dees}; we have (recalling \eqref{def:smooth-ord})
\[
\int_0^tW_s^{32s}(\xi)\frac{ds}{s}\leq 6\log(2)(\Ord_\phi(\xi,32t)-\Ord_\phi(\xi,0))
\]
which is in turn bounded by uniform upper bounds on the order. Note that this bound would immediately imply the desired result.

Now, in practice we are {\em not} always able to apply Lemma \ref{lem:mean-flat}; if $x$ is a point such that $W_{\rho r}^r(x)\geq\delta$, the lemma does not apply.  However, if $x$ is a point such that $\Ord_\phi(x,32)-\Ord_\phi(x,0)<\delta$, clearly any pinchings that appear will be small enough to apply Lemma \ref{lem:mean-flat}.  Hence, if $S_1:=\{x\in\mathcal{S}^k_{0,\eta}:\Ord_\phi(x,32)-\Ord_\phi(x,0)<\delta\}$, the work so far shows that $S_1$ is $k$-rectifiable. 

Now, if we let $S_i:=\{x\in\mathcal{S}^k_{0,\eta}:\Ord_\phi(x,32/i)-\Ord_\phi(x,0)<\delta\}$, by applying the above method and Theorem \ref{thm:at} to the measure $\mu_i:=\mathcal{H}^k\llcorner S_i$, the same reasoning shows that $S_i$ is $k$-rectifiable. 
 Finally, because $\lim_{r\to0}\Ord_\phi(x,r)=\Ord_\phi(x,0)$ for each $x$, the sets $S_i$ form a countable cover of $\mathcal{S}^k_{0,\eta}$. In particular, $\mathcal{S}^k_{0,\eta}$ is $k$-rectifiable.
\end{proof}
\appendix
\section{Standard Covering Arguments}\label{sec:arxivproof}

We still must prove Theorem \ref{thm:k-2nd-cover}.  We first prove an intermediate covering lemma, which relies on the following result.

\begin{theorem}[Appears as {\cite[Theorem 3.4]{nv17}}; cf. {\cite[Theorem 42]{nv18}}]\label{thm:nv-cover}
    There exist $\delta_R(n)>0$ and $C_R(n)$ such that the following holds.  Let $\{B_{r_j/5}(x_j)\}_{x_j\in S}\subseteq B_3(0)$ be a collection of pairwise disjoint balls with $x_j\in B_1(0)$, and let $\mu:=\sum_j\omega_kr_j^k\delta_{x_j}$ be the associated measure.  Assume that for each $B_r(y)\subseteq B_2$ we have
    \[
    \int_{B_r(y)}\left(\int_0^rD_\mu^k(z,s)\frac{ds}{s}\right)d\mu(z)\leq \delta_R^2r^k.
    \]
    Then we have the uniform bound
    \[
    \mu(B_1(0))=\omega_k\sum_{j}r_j^k\leq C_R(n).
    \]
\end{theorem}
We proceed to the intermediate covering lemma. The main difference from previously constructed arguments as in \cite{nv17} is in the proof of Claim \ref{claim:second claim}. We invoke two lemmas (one at small scales and one at large scales) to produce the desired result across all scales.  

\begin{proposition}\label{prop:initial-cover}
    Let $u:B_{64}(0) \to X_C$ satisfy \eqref{eq:Icond}, and let $0<\rho\leq\frac1{256}$ 
    and also let $\sigma,\tau$ be given such that $0<\sigma<\tau\leq\frac18$.  There exist $\delta=\delta(n,\rho,\eta,X_C,\Lambda)>0$ and $C=C(n)$ so that the following holds for any $0<\eta<\eta_0$, where $\eta_0(n)$ is chosen as in Lemma \ref{lem:hom-ord-bound}. 
    
    For any $x\in B_{1/8}(0)$, and any $D\subseteq \mathcal{S}^k_{0,\eta,\delta\sigma}(u)\cap B_\tau(x)$, let $M:=\sup_{y\in D}\Ord_\phi(y,8\tau)$.  There is a covering of $D$ by finitely many balls $B_{r_i}(x_i)$ so that
    \begin{enumerate}
        \item $r_i\geq10\rho\sigma$,
        \item $\sum_{i}(r_i)^k\leq C\tau^k$
        \item For each $i$, either $r_i\leq\sigma$, or the set
        \begin{equation}\label{eq:bad-pts}
            F_i:=D\cap B_{r_i}(x_i)\cap\{y:\Ord_\phi(y,\rho r_i)> M-\delta\}
        \end{equation}
        is contained in $B_{r_i}(x_i)\cap B_{\rho r_i}(L_i)$ for a $(k-1)$-dimensional affine subspace $L_i$.
    \end{enumerate}
\end{proposition}

\begin{proof}
    For convenience, we recenter so that $x=0$ and rescale so that $\tau=\frac18$. We observe that for any $y\in B_{1/8}(0)$, $\Ord_\phi(y,1)$ is bounded in terms of $\Lambda$ by Lemma \ref{lem:ben4.2}.  Throughout the proof, we shall impose a finite number of conditions on $\delta$ which allow us to construct the desired covering.  We will construct a covering of $D$ inductively, and then show that it satisfies the desired properties as long as $\delta$ is small enough.

    At the initial step, we simply choose $\mathcal C(0)=\{B_{1/8}(0)\}$.  At each step in the induction process, we will leave some balls of the cover unchanged (and refer to them as \emph{bad} balls), while the other (\emph{good}) balls in the cover will be refined.  At every step in the induction we want the following to hold:
    \begin{enumerate}
        \item[(A)] $D$ is covered by $\mathcal{C}(m):= \mathcal G(m) \cup \mathcal B(m)$ where
\begin{enumerate}
\item[(A1)] if $B_{r}(x) \in \mathcal B(m)$ then $r \geq \frac 18 (10\rho)^{m}$ and the set
  \[
    F(B_{r}(x)):=D\cap  \{y\in B_r(x):\Ord_\phi(y,\rho r)>M-\delta\}
    \]is contained in $B_{\rho r}(L)$ for some $(k-1)$-dimensional affine subspace $L$.
    \item[(A2)] if $B_{r}(x) \in \mathcal G(m)$ then $r=\frac 18(10 \rho)^m$ and the set
      \[
    F(B_{r}(x)):=D\cap  \{y\in B_r(x):\Ord_\phi(y,\rho r)>M-\delta\}
    \] $\rho r$-spans some $k$-dimensional affine subspace $V$.
\end{enumerate}
        \item[(B)] For any two balls $B_r(x),B_{r'}(x')\in \mathcal C(m)$, $B_{r/5}(x)\cap B_{r'/5}(x')=\emptyset$.
        \item[(C)] For all $B_r(x) \in \mathcal C(m)\setminus \mathcal C(0)$, $m \geq 1$, 
        \begin{enumerate}
            \item[(C1)] \label{item:pinchcenters} $\Ord_\phi(x,\rho r/5) \geq M-\epsilon$ (for an $\epsilon>0$ to be specified).
            \item[(C2)] $u$ is not $(\eta/2, 8s, k+1)$-homogeneous at $x$ whenever $8s \in [r/5, \tau/2]$.
        \end{enumerate}
    \end{enumerate}    Because we want all of our balls to have radius $r\geq 10\rho\sigma$, we stop as soon as $\frac18(10\rho)^m\leq\sigma$; we call the step when this happens $\nu$.  
    
    For the inductive construction, suppose that $B_{r}(x)\in\mathcal{G}(m)$.  
  Then $F(B_{r}(x))$ $\rho r$-spans some $k$-dimensional affine subspace $V$.  Choosing $\delta\leq \delta_{\ref{lem:cone-split}}(n,\rho, \eta, X_C, \Lambda)$, Lemma \ref{lem:cone-split} implies that $D\cap B_r(x)\subseteq B_{2\rho r}(V)$. (Note here that since $B_r(x) \cap \mathcal S_{0,\eta,\delta \sigma}^k(u) \neq \emptyset$, the hypothesis $0_X \in B_r(x)$ is satisfied.) For convenience denote the collection of good balls $\{B_i\}_{i\in I}=\mathcal G(m)$ and their corresponding $k$-dimensional subspaces $V_i$. Then clearly, with $r=\frac 18(10 \rho)^m$, 
    \[
D(m):=\bigcup_{i\in I}\left(  D\cap B_i\right)  \subset \bigcup_{i\in I}\left(D\cap B_{2\rho r}(V_i)\right).
    \]
    We now cover $D(m)$ by a collection of balls of radius $10\rho r=\frac18(10\rho)^{m+1}$ so that the corresponding balls of radius $2\rho r$ are pairwise disjoint and so that the centers of the balls are in $\bigcup_{i\in I}B_i\cap V_i$.  Call this collection of balls $\mathcal{F}(m+1)$; again, observe that all of these balls have radius $\frac18(10\rho)^{m+1}$.

    \begin{claim}\label{claim:first claim}
     
For any $\epsilon>0$, if $\delta=\delta(n,\rho, \eta,\epsilon, X_C,\Lambda)>0$ is sufficiently small, then for any $B_s(x)\in\mathcal F(m+1)$,
    \[
    \Ord_\phi(x,\rho s/5)\geq M-\epsilon.
    \]
  \end{claim}
  \begin{proof}There is a good ball $B=B_r(y)\in \mathcal C(m)$, with $s=10\rho r$ and a $k$-dimensional affine subspace $V$ which is $\rho r$-spanned by $F(B_r(y))$, and $x\in B_r(y)\cap V$.  Applying Lemma \ref{lem:ord-pinch}, as long as $0<\delta\leq \delta_{\ref{lem:ord-pinch}}(n,\rho,\epsilon,X_C,\Lambda)$,
    \begin{equation*}\label{eq:pinch-bound}
    \Ord_\phi(x,\rho s/5)>M-\epsilon.
    \end{equation*}
\end{proof}

\begin{claim}\label{claim:second claim}
For $\delta=\delta(n,\rho, \eta, X_C,\Lambda)>0$ small enough, for any $B_{s}(x) \in \mathcal F(m+1)$, $u$ is not $(\eta/3,t,k+1)$-homogeneous at $x$ for any $t\in [s/5,1/16]$.
\end{claim}  
\begin{proof}
First we choose $0<\epsilon \leq \delta_{\ref{lem:move-inhom}}(n, \rho,\eta,X_C,\Lambda)$ and $0<\delta< \min\{\delta_{\ref{lem:move-inhom}},\delta_{\ref{lem:move-inhom-large}},\delta_{\ref{claim:first claim}}\}$. 

    Since $B_{s}(x) \in \mathcal F(m+1)$, there exists $B_r(y) \in \mathcal G(m)$ such that $s=10\rho r$, $x\in B_r(y)$, and $F(B_r(y))$ $\rho r$-effectively spans an affine space $V$ (where here $F(B_r(y))$ has been defined by the chosen $\delta$). Since, by Claim \ref{claim:first claim}, $W^{1}_{\rho r}(x)<\delta_{\ref{lem:move-inhom}}$, we get the desired result immediately for $t \in [s/5, s/(5\rho)]$. 

    For scales $t\in[s/(5\rho),1/16]$, we invoke Lemma \ref{lem:move-inhom-large}, rescaling $B_t(x)$ to $B_1(0)$ instead of Lemma \ref{lem:move-inhom}, because for such $t$, the $y$ from the previous paragraph is in $B_t(x)$. We hence conclude that $u$ is not $(\eta/3,t,k+1)$-homogeneous for any of these $t$, as well.
\end{proof}

\begin{remark}
    Note that the above claims will imply (C1) and (C2) since, as we will see below, every $B_r(x) \in \mathcal C(m)\setminus \mathcal C(0)$ will come from the set $\mathcal F(j)$, for some $1 \leq j \leq m$. 
\end{remark}
    We now construct the collection $\mathcal{C}(m+1)$.  Let $\mathcal{B}_{1/5}(m)$ denote the balls with the same centers as the balls in $\mathcal{B}(m)$ and $\frac15$ the radii; by the inductive hypothesis this collection is pairwise disjoint.  We then define $\mathcal{F}'(m+1)\subset\mathcal{F}(m+1)$ where $B_r(x)\in\mathcal{F}'(m+1)$ if it does not intersect any element of $\mathcal{B}_{1/5}(m)$, and discard it otherwise.  We then define $\mathcal{C}(m+1)=\mathcal{B}(m)\cup\mathcal{F}'(m+1)$.  We still must check that $\mathcal{C}(m+1)$ is a cover of $D$.

    Any point of $D$ that is either in one of the balls of $\mathcal{B}(m)$ or one of the retained balls of $\mathcal{F}'(m+1)$ is covered.  So, suppose that $z\in D$ lies in some $B_r(x)\in\mathcal F(m+1)$ which is {\em not} an element of $\mathcal C(m+1)$.  This means that there exists some $B_{r'}(x')\in \mathcal B(m)$ so that $B_r(x)\cap B_{r'/5}(x')\neq\emptyset$.  However, because $\frac{r'}{10}>r$, $z\in B_r(x)\subseteq B_{r'}(x')$ and thus $\mathcal C(m+1)$ covers $D$.

    Finally, we divide $\mathcal{C}(m+1)$ into $\mathcal{B}(m+1)$ and $\mathcal{G}(m+1)$.  For $B_r(x)\in\mathcal{B}(m)\subset\mathcal{C}(m+1)$, we of course keep $B_r(x)\in\mathcal{B}(m+1)$.  For $B_r(x)\in\mathcal{F}'(m+1)$, if
    \[
    F(B_r(x)):=D\cap\{y\in B_r(x):\Ord_\phi(y,\rho r)>M-\delta\}
    \]
    $\rho r$-spans a $k$-dimensional affine subspace $V$, we assign $B_r(x)$ to $\mathcal{G}(m+1)$.  Otherwise, we have that $F(B_r(x))$ is contained in $B_{\rho r}(L)$ for some $(k-1)$-dimensional affine subspace $L$, and we assign $B_r(x)$ to $\mathcal{B}(m+1)$.

    The collection $\mathcal C(\nu)$ will be the desired covering, where we recall that $\frac 18(10\rho)^\nu\leq\sigma$. Each cover $\mathcal C(m)$ satisfies points (A), (B), (C) within the proof, and $\mathcal C(\nu)$ clearly satisfies points (1) and (3) from our proposition.  We now verify the packing estimate (2).
    \begin{remark}
        If $\mathcal C(\nu)=\{B_{1/8}(0)\}$, then the packing estimate is trivially satisfied. So in all of what follows we will assume that $\mathcal C(\nu)\neq\{B_{1/8}(0)\}$.
    \end{remark}

   For convenience, we now write our cover as $\mathcal C(\nu)=\{B_{5s_i}(x_i)\}$ so that the balls $B_{s_i}(x_i)$ are pairwise disjoint.  To verify the packing estimate (2), we will show that
    \[
    \sum_{i}(s_i)^k\leq C(n).
    \]
    To do this, we consider measures $\mu$ of the type controlled by Theorem \ref{thm:nv-cover} and set
    \[    \mu:=\sum_i(s_i)^k\delta_{x_i}\hspace{5mm}\text{and}\hspace{5mm}\mu_s:=\sum_{i:s_i\leq s}(s_i)^k\delta_{x_i}.
    \]
    From the definition of $\mu_s$, we have:
    \begin{itemize}
        \item If $s\leq s'$, then $\mu_s\leq\mu_{s'}$,
        \item $\mu_{1/40}=\mu$,
        \item For $r<\overline{r}:=\frac1{40}(10\rho)^\nu$, $\mu_r=0$.
    \end{itemize}We shall show that for any $m\in \mathbb N$ with $s=2^m\overline{r}$ satisfying $s<\frac1{64}$ we have
    \begin{equation}\label{eq:inductive-bound}
    \mu_s(B_s(x))\leq C_R(n)s^k 
    \end{equation}
    for all $x$, where $C_R(n)$ is the constant which appears in Theorem \ref{thm:nv-cover}.

    As we are assuming $\mathcal C(\nu)\neq \{B_{1/8}(0)\}$, the $s_i$ are at most $\frac{10\rho}{40}\leq\frac1{128}$.  The estimate \eqref{eq:inductive-bound} then implies that for any $x$, $\mu(B_{1/128}(x))\leq C_R(n)$ so that covering $B_1(0)$ by finitely many such balls, we obtain the desired estimate, $\mu(B_1(0))\leq C(n)$.

    We establish \eqref{eq:inductive-bound} by induction on $m$.  In the base case, we observe that \[\mu_{\overline{r}}(B_{\overline{r}}(x))=N(x,\overline{r})\overline{r}^k,\] where $N(x,\overline{r})$ is the number of balls $B_{s_i}(x_i)$ in the cover so that $x_i\in B_{\overline{r}}(x)$ with $s_i=\overline{r}$.  Since the $B_{s_i}(x_i)$ are pairwise disjoint, and all are contained in $B_{2\overline{r}}(x)$, there are at most $2^n$ such balls.

    The remainder of the proof is the inductive step, showing that if \eqref{eq:inductive-bound} holds for $m$, it also holds for $m+1$.  This is accomplished in two steps---a crude bound, and the full inductive bound.  We write $r=2^m\overline{r}$, so that we are assuming that $\mu_r(B_r(x))\leq C_R(n)r^k$ for every $x$, and we wish to show that $\mu_{2r}(B_{2r}(x))\leq C_R(n)(2r)^k$, again for every $x$.

    \begin{claim}\label{claim:crude-bound} (Crude Inductive Packing Bound) By induction, we have that \begin{equation*} \mu_{2r}(B_{2r}(x))\leq C(n)C_R(n)(2r)^k.\end{equation*}  
    
    \end{claim}
    \begin{proof}
        
    This is quick because $\mu_{2r}=\mu_r+\sum_{i:r<s_i\leq2r}(s_i)^k\delta_{x_i}=:\mu_r+\overline{\mu_r}$, and we can crudely bound each of these summands separately to bound $\mu_{2r}(B_{2r}(x))$.  For $\mu_r$, we cover $B_{2r}(x)$ by $C(n)$ balls $B_r(y)$, and observe that on each of these we have the bound $\mu_r(B_r(y))\leq C_R(n)r^k$.  For $\overline{\mu_r}(B_{2r}(x))$, we have that $\overline{\mu_r}(B_{2r}(x))\leq N(x,2r)(2r)^k$, where $N(x,2r)$ counts those $B_{s_i}(x_i)$ in the cover for which $x_i\in B_{2r}(x)$ and $r<s_i\leq2r$.  Because these balls are all contained in $B_{4r}(x)$, and all have radius at least $r$, $N(x,2r)\leq C(n)$ as well.

\end{proof}

    To prove \eqref{eq:inductive-bound}, we shall apply (a rescaled version of) Theorem \ref{thm:nv-cover} to the measure $\overline{\mu}:=\mu_{2r}\llcorner B_{2r}(x)$.  In particular, if we can show that for any $y\in B_{2r}(x)$ and any $0<t\leq 2r$,
    \begin{equation}\label{eq:inductive-goal}
    (I):=\int_{B_t(y)}\int_0^t{D_{\overline{\mu}}^k(z,s)}\,\frac{ds}{s}\,d\overline{\mu}(z)<\delta_R^2t^k
    \end{equation}
    (where $\delta_R$ is the constant from Theorem \ref{thm:nv-cover}), we can then conclude that $\mubar(B_{2r}(x))\leq C_R(2r)^k$, which is what we want.
 \begin{claim}For $x_i\in\text{supp}(\mu)$, define
    \begin{equation}
    \Wbar_s(x_i):=\begin{cases}
W_s^{32s}(x_i)&\text{if }s\geq s_i\\
0&\text{otherwise.}
\end{cases}
    \end{equation}
Then for all $i$ and $0<8s\leq \frac 1{16}$,
    \begin{equation}
    D_{\mubar}^k(x_i,s)\leq Cs^{-k}\int_{B_s(x_i)}\Wbar_s(\zeta)\,d\mubar(\zeta),
    \end{equation}
    where the $C=C(n,\rho,\eta,X_C,\Lambda)$ that appears here is the $C$ from Lemma \ref{lem:mean-flat}.
\end{claim}
\begin{proof}
    If $s<s_i$, this is the inequality $0\leq0$. Now suppose $s_i \leq s$. By Claim \ref{claim:first claim}, 
    \[
    W_{8\rho s}^{16s}(x_i)\leq W_{8\rho s_i}^{1}(x_i)<\epsilon.
    \]
    Further, by Claim \ref{claim:second claim},  $u$ is not $(\eta/3,8s,k+1)$-homogeneous at $x$ since $8s \in [s_i, 1/16]$.  Now if $0< \epsilon \leq \delta_{\ref{lem:mean-flat}}$, we are free to apply Lemma {\ref{lem:mean-flat}} with $r=8s$ and the result immediately follows. 

\end{proof}
Note that in the previous claim, we again had to (possibly) decrease our choice of $\delta$ to achieve Claim \ref{claim:first claim} for the desired $\epsilon$.
    
    Hence, to bound $(I)$ in \eqref{eq:inductive-goal}, it suffices to bound
    \begin{align*}
    \int_{B_t(y)}&\int_0^ts^{-k-1}\int_{B_s(z)}\Wbar_s(\zeta)\,d\mubar(\zeta)\,ds\,d\mubar(z)=\\
    &\int_0^ts^{-k-1}\int_{B_t(y)}\int_{B_s(z)}\Wbar_s(\zeta)\,d\mubar(\zeta)\,d\mubar(z)\,ds.
    \end{align*}
    We may freely intersect both domains of integration with $B_{2r}(x)$, after which both of the integrals with respect to $d\mubar$ can be taken with respect to $d\mu_s$ instead.  Indeed, if $\zeta\in\text{supp}(\mubar)\setminus\text{supp}(\mu_s)$, $\zeta=x_i$ for some $B_{s_i}(x_i)$ for which $s<s_i$, in which case $\Wbar_s(\zeta)=0$ by definition.  For the middle integral, if $z\in\text{supp}(\mubar)\setminus\text{supp}(\mu_s)$, then $z=x_i$ with $s_i>s$, so there are no points in $B_s(z)\cap\text{supp}(\mubar)$ besides $z$ itself, and moreover $\Wbar_s(z)=0$ so the innermost integral simply vanishes for such $z$. 

Therefore, our aim is in fact to bound
    \[
 \int_0^ts^{-k-1}\int_{B_t(y)}\int_{B_s(z)}\Wbar_s(\zeta)\,d\mu_s(\zeta)\,d\mu_s(z)\,ds.
    \]
    This is where we will leverage the crude bound on $\mu_{2r}$.  Changing the order of integration,
    \begin{align*}
    (I)&\leq C\int_0^ts^{-k-1}\int_{B_{t+s}(y)}\Wbar_s(\zeta)\int_{B_s(\zeta)}\,d\mu_s(z)\,d\mu_s(\zeta)\,ds\\
    &\leq C'\int_0^ts^{-k-1}s^k\int_{B_{t+s}(y)}\Wbar_s(\zeta)\,d\mu_s(\zeta)\,ds
    \end{align*}
    where we have used \eqref{eq:inductive-bound} for $s\leq r$ and Claim \ref{claim:crude-bound} and the induction hypothesis for $s>r$, because $t\leq2r$ by assumption.  We change order of integration again and see that
    \begin{equation}\label{eq:penult-bound}
    (I)\leq\int_{B_{2t}(y)}\int_0^t\Wbar_s(\zeta)\,\frac{ds}{s}\,d\mu_t(\zeta)
    \end{equation}
    and we now fix some $\zeta\in\text{supp}(\mu_t)$, so that $\zeta=x_i$ for some $i$, and recall that $\Wbar_s(x_i)=\Ord_\phi(x_i,32s)-\Ord_\phi(x_i,s)$ as long as $s_i<s$ and $\Wbar_s(x_i)=0$ otherwise.  Letting $N$ be the largest natural so that $2^Ns\leq t$, we know that $32\cdot2^{N+1}s\leq1$ because $t\leq2r\leq\frac1{64}$.  We then compute
    \begin{align*}
    \int_0^t\Wbar_s(x_i)\,\frac{ds}s&=\int_{s_i}^{t}\Wbar_s(x_i)\,\frac{ds}s=\int_{s_i}^t(\Ord_\phi(x_i,32s)-\Ord_\phi(x_i,s))\,\frac{ds}s\\
    &\leq\sum_{j=0}^N\int_{2^js_i}^{2^{j+1}s_i}(\Ord_\phi(x_i,32s)-\Ord_\phi(x_i,s))\,\frac{ds}s\\
    &\leq\sum_{j=0}^N(\Ord_\phi(x_i,32\cdot2^{j+1}s_i)-\Ord_\phi(x_i,2^js_i))\int_{2^js_i}^{2^{j+1}s_i}\frac{ds}s\\
    &=\log(2)\sum_{j=0}^N\Ord_\phi(x_i,2^{6+j}s_i)-\Ord_\phi(x_i,2^js_i)\\
    &=\log(2)\sum_{\ell=0}^5\sum_{j=0}^N\Ord_\phi(x_i,2^{\ell+j+1}s_i)-\Ord_\phi(x_i,2^{\ell+j}s_i)\\
    &=\log(2)\sum_{\ell=0}^5\Ord_\phi(x_i,2^{\ell+N+1}s_i)-\Ord_\phi(x_i,2^\ell s_i)\\
    &\leq6\log(2)(\Ord_\phi(x_i,1)-\Ord_\phi(x_i,s_i))\leq6\log(2)\epsilon
    \end{align*}

    Again using either \eqref{eq:inductive-bound} and a covering argument, or using Claim \ref{claim:crude-bound}, we see that $\mu_t(B_{2t}(y))\leq Ct^k$, where $C$ is a purely dimensional constant.  Hence, combining these two facts with \eqref{eq:penult-bound} we have that
    \[
    \int_{B_t(y)}\int_0^t{D_{\overline{\mu}}^k(z,s)}\,\frac{ds}{s}\,d\overline{\mu}(z)\leq C(n,\rho,\eta,X_C,\Lambda)\epsilon t^k.
    \]
    As long as $\epsilon=\epsilon(n,\rho,\eta,X_C,\Lambda)$ is sufficiently small (which requires a suitable choice of $\delta=\delta(n,\rho,\eta,X_C,\Lambda)$), we apply Theorem \ref{thm:nv-cover} to conclude that $\mubar(B_{2r}(x))\leq C_R(n)(2r)^k$, as desired.
\end{proof}
\begin{remark}
The above computation of $\int_0^t\Wbar_s(x_i)\frac{ds}s$ is essentially the same as the omitted computation in the proof of Theorem \ref{thm:k-rect-local}. 
\end{remark}
Using Proposition \ref{prop:initial-cover}, we now prove Theorem \ref{thm:k-2nd-cover}, following essentially identically the framework of \cite{nv17,nv18,dmsv}.

\begin{proof}[Proof of Theorem \ref{thm:k-2nd-cover}]
    We shall construct this cover by repeated applications of Proposition \ref{prop:initial-cover}, where we will fix the $\rho$ of that Proposition during this proof.  For convenience we recenter so $x=0$ and rescale so that $S=\frac 18$.

    First, we apply Proposition \ref{prop:initial-cover} with $\tau=S=\frac 18$ and $\sigma=s$.  Call the covering obtained in this manner $\mathcal{C}(0)$, and partition $\mathcal{C}(0)$ into the good balls $\mathcal{G}(0):=\{B_{r_y}(y):r_y\leq s\}$ and the bad balls $\mathcal{B}(0):=\{B_{r_y}(y):r_y>s\}$.  For each $B_{r_y}(y)\in\mathcal{B}(0)$, we know that the set $F_y=D\cap\{z\in B_{r_y}(y):\Ord_\phi(z,\rho r_y)>M-\delta\}$ is contained in $B_{\rho r_y}(L_y)$ for some $(k-1)$-dimensional affine subspace $L_y$.  Cover the set $B_{2\rho r_y}(L_y)\cap B_{r_y}(y)$ by $N=C(n)\rho^{1-k}$ balls of radius $4\rho r_y$, where $C(n)$ is a dimensional constant.  Call the collection of all balls obtained in this manner $\mathcal{C}(1)$.  We subdivide $\mathcal{C}(1)$ in a similar manner to the division of $\mathcal{C}(0)$, setting $\mathcal{G}(1):=\{B_{r_y}(y)\in\mathcal{C}(1):r_y\leq s\}$, and $\mathcal{B}(1):=\{B_{r_y}(y)\in\mathcal{C}(1):r_y>s\}$. We observe that
    \[
    \sum_{B_{r_y}(y)\in\mathcal{C}(1)}r_y^k\leq C(n)\rho^{1-k}\sum_{B_{r_y}(y)\in\mathcal{C}(0)}(\rho r_y)^k=C(n)\rho\sum_{B_{r_y}(y)\in\mathcal{C}(0)}r_y^k
    \]
    by construction.  At this point, we fix $\rho$ so small that $C(n)\rho\leq\frac12$ and use this fixed constant value for the remainder of the argument.

    Repeating this construction some finite number of times, eventually we reach a collection $\mathcal{C}(\ell)$ in which all of the balls $B_{r_y}(y)$ are of radius $r_y\leq s$ (since at each stage the maximum radius of a ``bad ball" is multiplied by $4\rho<1$).  Consider now the cover which is the union of all of these collections $\mathcal{C}:=\bigcup_{i=0}^\ell\mathcal{C}(i)$.  For convenience, we also define the collections of centers of these balls, $K(i)$, by letting $y\in K(i)$ if $y$ is the center of a ball in $\mathcal{C}(i)$.

    We now define the sets $D_y\subseteq D$ corresponding to this cover.

    \begin{itemize}
        \item For $B_{r_y}(y)\in\mathcal{G}(i)$, we set $D_y=B_{r_y}(y)\cap D$ and observe that $r_y\leq s$
        \item For $B_{r_y}(y)\in\mathcal{B}(i)$, we set $D_y=(B_{r_y}(y)\cap D)\setminus F_y$, and observe that, by construction, for each $z\in D_y$, $\Ord_\phi(z,\rho r_y)\leq M-\delta$.
    \end{itemize}

    We first show that the sets $D_y$ cover $D$.  For each $i=0,\dots,\ell$, we have by construction that \[\bigcup_{y\in K(i)} F_y\subseteq\bigcup_{z\in K(i+1)}B_{r_z}(z).\]
    Hence, we have that
    \[
    D\subseteq\bigcup_{y\in K(0)}D_y\cup\bigcup_{z\in K(1)}B_{r_z}(z)
    \]
    and applying the previous observation inductively (noting that the $F_z$ for the bad balls of $\mathcal{C}(1)$ are covered by the collection $\mathcal{C}(2)$, and so on), we see that $D$ is covered by the $D_y$, as desired.

    At this point, for every $y$, either $r_y\leq s$, or
    \[
    \sup_{z\in D_y}\Ord_\phi(z,\rho r_y)\leq M-\delta
    \]
    which differs from the desired \eqref{eq:ord-drop} by a factor of $\frac\rho8$ in the second argument of $\Ord_\phi$.  To correct for this, we simply cover these balls $B_{r_y}(y)$ by a family of balls $B_{\rho r_y/8}(x)=:B_{r_x}(x)$, and define $D_x=B_{r_x}(x)\cap D_y$.  Since we have defined $\rho$ purely in terms of the dimension $n$, the packing bound is worsened only by a dimensional constant $C(n)$, and \eqref{eq:ord-drop} holds on this new collection.

    Finally, some of the balls in our cover may have radius less than $s$, but all of them have radius at least $\rho s$, so replacing all of these balls by concentric balls of radius $s$ again worsens the packing bound by a dimensional constant.
\end{proof}

\end{document}